\documentclass[11.5pt]{article}
\usepackage{lipsum}

\usepackage[margin=1in]{geometry}
\usepackage{fancyhdr}
\usepackage{graphicx}
\usepackage{mathtools}
\usepackage{amsmath, amssymb, amsfonts, amsbsy, amsthm, latexsym, color}
\usepackage{ulem}
\usepackage{pifont}
\newcommand{\cmark}{\ding{51}}%
\newcommand{\xmark}{\ding{55}}%
\usepackage{mathtools}
\usepackage{mathrsfs}
\usepackage{makecell}
\usepackage{cancel}

\usepackage{hyperref} 
\usepackage{empheq}

\usepackage[T1]{fontenc}
\usepackage{authblk}
\usepackage{graphicx}
\usepackage{enumerate}
\usepackage{caption}
\usepackage{subcaption}
\newcommand{\RN}[1]{%
  \textup{\uppercase\expandafter{\romannumeral#1}}%
}
\newcommand{\rom}[1]{\uppercase\expandafter{\romannumeral #1\relax}}
\everymath{\displaystyle}
\usepackage{float}

\usepackage[ruled,vlined]{algorithm2e}
\newtheorem{theorem}{Theorem}[section]
\newtheorem{proposition}[theorem]{Proposition}
\newtheorem{lemma}[theorem]{Lemma}
\newtheorem{corollary}[theorem]{Corollary}

\theoremstyle{definition}
\newtheorem{definition}{Definition}[section]

\newtheorem{remark}[theorem]{Remark}

\numberwithin{equation}{section}

\newcommand{\C}{{\mathbb{C}}}

\usepackage{romannum}
\newcommand{\rank}{\text{rank}}

\begin{document}
\pagenumbering{arabic}
\title{An exact sin$\Theta$ formula for matrix perturbation analysis and its applications}
\author{He Lyu, Rongrong Wang}
\affil{Michigan State University}

%
%
%


\date{\today}

\maketitle

\begin{abstract}
In this paper, we establish a useful set of formulae for the $\sin\Theta$ distance between the original and the perturbed singular subspaces. These formulae explicitly show that how the perturbation of the original matrix propagates into singular vectors and singular subspaces, thus providing a direct way of analyzing them. Following this, we derive a collection of new results on SVD perturbation related problems, including a tighter bound on the $\ell_{2,\infty}$-norm of the singular vector perturbation errors under Gaussian noise,  a new stability analysis of the Principal Component Analysis and an error bound on the hard singular value thresholding operator. For the latter two, we consider the most general rectangular matrices with full matrix rank.

\end{abstract}
\section{Introduction}
Singular value decomposition (SVD) is a fundamental tool in computational mathematics. Many widely used algorithms in numerical analysis and statistics (e.g., principal component analysis \cite{pearson1901liii,cai2021subspace,abbe2022}, matrix completion \cite{candes2012exact,candes2010matrix,keshavan2009matrix}, matrix denoising \cite{donoho2014minimax,gavish2014optimal}, community detection \cite{yun2014accurate,chin2015stochastic,abbe2017community}, {graph inference \cite{tang2018limit,athreya2021estimation},} etc.) involve the SVD computation. Since the singular vectors and singular subspaces can be sensitive to noise, a careful study of the stability of  SVD is necessary. 

For a given matrix $A$, let $A=U\Sigma V^T$ be its SVD, and $\widetilde{A}=\widetilde{U}\widetilde{\Sigma}\widetilde{V}^T$ be the SVD of the noisy version $\widetilde{A}=A+ \Delta A$. In perturbation analysis, our goal is to characterize the robustness of the left or right singular subspaces under an arbitrary perturbation $\Delta A$.

Classical subspace perturbation results, including Davis-Kahan's theorem \cite{davis1970rotation}, Wedin's $\sin\Theta$ theorem \cite{weyl1912asymptotische} and many others (e.g., \cite{stewart2006perturbation,dopico2000note,dopico2002perturbation}), provide bounds on the  sin$\Theta$ angles between original and perturbed subspaces. For general symmetric matrices with no statistical assumption on the noise,  Davis-Kahan's $\sin\Theta$ bound is already tight and easy to use. However, when the perturbation matrix $\Delta A$ is random, Davis-Kahan's bound becomes a random quantity. To solve this problem, 
 deterministic variants of the Davis-Kahan's $\sin\Theta$ theorem were introduced in \cite{yu2015useful,cai2018rate} that are particularly useful for statistical applications. {Additionally, various asymptotic bounds on eigenvector perturbations have also been derived in \cite{cape2019signal,tang2018limit,fan2022asymptotic,agterberg2022entrywise}.}

Recently, it was noticed that in addition to  norms of $\sin\Theta$ angles, one may derive perturbation bound under other metrics, which can bring extra benefits. One such metric is the $\ell_{2,\infty}$-norm of the difference between the original and perturbed singular vectors. In applications such as clustering, classification, and dimension reduction,  the $\ell_{2,\infty}$ metric is more accurate in the sense {that it provides finer entry-wise error bounds} of the embedded data. In addition, bounds on the $\ell_{2,\infty}$-norm are possible to be much smaller than those on the $\sin\Theta$ angle \cite{cape2019two}, which is another benifit of using it. Recently, many interesting results have been derived along this direction, including 
\cite{abbe2020entrywise,chen2021asymmetry,cheng2020inference,cape2019two,abbe2022,cai2021subspace}, but the problem is still open. 
 
Besides the perturbation bounds on individual factors (i.e., $U$, $\Sigma$, $V$) in the SVD, bounds on combinations of the factors are also desired by many applications. For example, in PCA, the target quantity is {the PC scores $U\Sigma$. Several works related to perturbation of this quantity exist. In  \cite{blanchard2007statistical},  the perturbation error of eigenspaces of kernel PCA is studied. In \cite{abbe2022}, an $\ell_p$ stability analysis of PC scores is developed for the hollowed version of PCA, which is the normal PCA with the diagonal entries of the Gram matrix  removed before carrying out the SVD. To the best of our knowledge, a tight bound for the PC scores of the vanilla PCA is still missing. Another example that requires perturbation analysis on combinations of factors is  singular value truncation \cite{tanner2013normalized,donoho2014minimax,cai2010singular}, where the target quantity whose perturbation we care about is the best rank-$r$ approximation, $A_r$, of $A$. Previously, sharp perturbation bounds on $A_r$ only exist for low-rank matrices {\cite{luo2021schatten},} and that for general full-rank matrices is still missing.} 
 


In this paper, we present several new perturbation results including an improved $\ell_{2,{\infty}}$-norm bound on the singular vector perturbation under Gaussian noise,  a perturbation bound for the hard singular value thresholding operator applied to full-rank matrices, and a useful error bound for the perturbation of the PC scores. These new results are either derived or motivated by the new set of sin$\Theta$ expressions we shall present in Section \ref{sec:3}.


\section{Collection of new perturbation results on SVD and its derivatives}\label{sec:2}
Before stating our main mathematical tool in Section \ref{sec:3}, we first present the three aforementioned implications, as they might be of independent interests.
\subsection{A new bound on $\ell_{2,\infty}$-norm of the singular vector perturbation}\label{sec:2.1}

Deriving tight $\ell_{2,\infty}$-norm bounds for the singular vector perturbation is an active research area in statistics  \cite{abbe2020entrywise,chen2021asymmetry,cheng2020inference,cape2019two}.
For many machine learning tasks (e.g., spectral clustering and Principal Component Analysis),  the $\ell_{2,\infty}$-norm provides a better characterization of the embedding quality  than the $\sin\Theta$ angle as it is a point-wise metric reflecting the error of individual embedding. {Previous literature has presented various bounds on the $\ell_{2,\infty}$-norm of singular vector perturbation. To list a few, \cite{abbe2020entrywise} established an $\ell_{2,\infty}$-norm  bound for eigenspaces of symmetric random matrices whose expectations are of low-rank. The result is shown to be useful for analyzing spectral methods on the stochastic block model. \cite{abbe2022} developed an $\ell_{2,p}$ analysis for a hollowed  PCA, for any $1\leq p\leq \infty$. Hallowed PCA in the presence of missing data has also been investigated in \cite{cai2021subspace}, with a special focus in the case where the number of features is significantly larger than the number of samples. Our problem setting and approach are more closely related to \cite{cape2019two}, which established an entry-wise singular subspace perturbation bound for low-rank matrices through a Procrustean matrix decomposition.} 

{In this section, we are interested in the case where each entry in the perturbation matrix follows i.i.d. Gaussian distribution. Statistical applications that fall into this regime include Gaussian Mixture Model with isotropic noise covariance matrix \cite{loffler2021optimality}.}

To facilitate the illustration, we introduce some notation. For rectangular matrices $A, \ \Delta A \in \mathbb{R}^{n\times m}$, we can write the conformal SVD of the original matrix $A$ and its perturbed version $\widetilde A=A+\Delta A$ as 
\begin{equation}\label{eq:conformal}
A=U\Sigma V^T=\begin{pmatrix}
U_1 & U_2
\end{pmatrix}
\begin{pmatrix}
\Sigma_1&\ \\
\ & \Sigma_2 
\end{pmatrix}
\begin{pmatrix}
V_1^T \\
V_2^T
\end{pmatrix},\
\widetilde A=\widetilde U\widetilde\Sigma \widetilde V^T=\begin{pmatrix}
\widetilde U_1 & \widetilde U_2
\end{pmatrix}
\begin{pmatrix}
\widetilde \Sigma_1&\ \\
\ & \widetilde \Sigma_2 
\end{pmatrix}
\begin{pmatrix}
\widetilde V_1^T \\
\widetilde V_2^T
\end{pmatrix}.
\end{equation}
Here $U_1\in\mathbb{R}^{n\times r},\ U_2\in\mathbb{R}^{n\times (n-r)},\ V_1\in\mathbb{R}^{m\times r},\ V_2\in\mathbb{R}^{m\times (m-r)}$,\ {$[U_1,U_2]\in\mathbb{R}^{n\times n}, [V_1,V_2]\in\mathbb{R}^{m\times m}$ are orthogonal matrices}, $\Sigma_1=\text{diag}\{\sigma_1,\sigma_2,...,\sigma_r\}\in\mathbb{R}^{r\times r}$, $ \Sigma_2=\text{diag}\{\sigma_{r+1},\sigma_{r+2},...,\sigma_{\min\{m,n\}}\}\in\mathbb{R}^{(n-r)\times (m-r)}$, {and the singular values are indexed in non-increasing order, i.e., $\sigma_1\geq\sigma_2\geq\cdots\geq\sigma_r\geq\sigma_{r+1}\geq\cdots\geq\sigma_{\min\{m,n\}}$.} When $n\neq m$, $\Sigma_2$ is rectangular, and the extra columns/rows are padded with $0$s. The decomposition of $\widetilde A$ has a similar structure with {non-increasing singular values} $\widetilde\Sigma_1=\text{diag}\{\widetilde\sigma_1,\widetilde\sigma_2,...,\widetilde\sigma_r\}\in\mathbb{R}^{r\times r},\ \widetilde\Sigma_2=\text{diag}\{\widetilde\sigma_{r+1},\widetilde\sigma_{r+2},...,\widetilde\sigma_{\min\{m,n\}}\}\in\mathbb{R}^{(n-r)\times (m-r)}$. 

For a matrix $A$, the $\ell_{2,\infty}$-norm is defined as 
\[
\|A\|_{2,\infty} : = \sup\limits_{\|x\|_2=1} \|Ax\|_{\infty}.
\]
One can show that $\|A\|_{2,\infty} = \max_{i} \|a_i\|_2$, where $a_i$ is the $i$th row of $A$.

In dimensionality reduction, to characterize the difference between $\widetilde{U}_1$ and $U_1$, a desirable measure is $\min_{Q\in \mathbb{O}_r}\|\widetilde{U}_1-U_1Q\|_{2,\infty}$, where $Q$ is a rotation matrix and $\mathbb{O}_r$ is the orthogonal matrix group in dimension $r$. In other words, we consider the difference between $\widetilde{U}_1$ and $U_1$ after they are maximally aligned by a proper rotation $Q$.

A well known property about $\min_{Q\in \mathbb{O}_r}\|\widetilde{U}_1-U_1Q\|_{2,\infty}$ is that it is smaller than a constant multiple of the $\sin\Theta$ angle \cite{cai2018rate}
\begin{equation} \label{eq:trivial}
\min_{Q\in \mathbb{O}_r}\|\widetilde{U}_1-U_1Q\|_{2,\infty} \leq \min_{Q\in \mathbb{O}_r}\|\widetilde{U}_1-U_1Q\| \leq  \sqrt{2}\|\sin\Theta(U_1,\widetilde{U}_1)\|.
\end{equation} 
This provides a trivial bound on the $\ell_{2, \infty}$-norm error, but can be very pessimistic.
To see why, we need the following definition of incoherent matrices.
\begin{definition}[Incoherent] \label{def:inco} A matrix $U \in \mathbb{R}^{n\times r}$ {with orthonormal columns} $(n\geq r)$ is said to be
$\mu$-incoherent $(\mu\geq 1)$ if $\|U\|_{2,\infty}\leq \mu\sqrt{\frac{r}{n}}$.
\end{definition}
Suppose the perturbation matrix $\Delta A$ has  i.i.d. Gaussian entries $N(0,\sigma^2)$, and the matrix $U_1$ of leading singular vector is $\mu$-incoherent. Then the bound \eqref{eq:trivial} is pessimistic in that it {only gives a bound of $O(1)$, while the bound we are about to provide is $O\left(\frac{r}{\sqrt n}\right)$}. 
Explicitly, \eqref{eq:trivial} combined with  Wedin's sin$\Theta$ bound yields that, provided $$c_1\sigma \sqrt{\max\{n,m\}} \leq \sigma_r(A)-\sigma_{r+1}(A), \quad \textrm{(gap condition)}$$  it holds with high probability that,
\begin{equation}\label{eq:l2infnaive}
\min_{Q\in \mathbb{O}_r}\|\widetilde{U}_1-U_1Q\|_{2,\infty} \leq \sqrt{2}\min\left\{1,\frac{c_2\sqrt {\max\{n,m\}}\sigma}{\sigma_r-\sigma_{r+1}}\right\}\sim O(1). 
\end{equation}
Here $c_1,c_2$ are absolute constants, and the big O notation is with respect to the size variables $n$ and $m$.
The gap condition implies that the noise level $\sigma$ can be as large as $O(1/\sqrt{\max\{n,m\}})$, so the order of  $\sigma\sqrt{\max\{n,m\}}$ is $O(1)$. Hence the bound in \eqref{eq:l2infnaive} is $O(1)$. 


We show that \eqref{eq:l2infnaive} is pessimistic by deriving a bound that is {of order $O(\frac{r}{\sqrt n})$}.

\begin{theorem}\label{thm:2toinf_full} Suppose $\widetilde{A}=A+\Delta A \in \mathbb{R}^{n\times m}$, $\bar n :=\max\{n,m\}$, $\Delta A$ has i.i.d. $N(0,\sigma^2)$ entries, and assume $\sigma_r(A)-\sigma_{r+1}(A)>21\sigma\sqrt{ \bar n}$. Then with probability {at least} $1-\frac{c}{n^2}$,
\begin{equation}\label{eq:thm}
\min_{Q\in \mathbb{O}_r}\|\widetilde{U}_1 -U_1Q \|_{2,\infty} \leq c_1\|U_1\|_{2,\infty}\frac{\sigma^2 \bar n}{(\sigma_r( A)-\sigma_{r+1}(A))^2}+c_2\sigma\frac{R(r, n)}{\sigma_r( A)-\sigma_{r+1}( A)},
\end{equation}
where 
$$
R(r,n) = \left\{\begin{aligned}
&\sqrt r+\sqrt{\log  n},\quad & \textrm{if $A$ is rank $r$}; \\
& r+\sqrt{r \log n}, & \textrm{else},
\end{aligned}\right.
$$
and {$c,\ c_1,\ c_2, \sigma$ are absolute constants independent of $n$ and $m$. }
\end{theorem}
The following corollary of Theorem \ref{thm:2toinf_full} may be easier to digest.
\begin{corollary}\label{col:2toinf_full} Under the same assumption as in Theorem \ref{thm:2toinf_full}, if we additionally assume that the matrix $U_1$ holding the leading $r$ left singular vectors of $A$ is $\mu_1$-incoherent {with some constant $\mu_1$}, then with probability {at least} $1-\frac{c}{n^2}$,
\[
\min_{Q\in \mathbb{O}_r}\|\widetilde{U}_1-U_1Q\|_{2,\infty} \leq C \cdot \frac{\mu_1 \sqrt r+\sqrt{r\log n}+r}{\sqrt n} \sim O\left(\frac{r}{\sqrt n}\right),
\]
where $c$ and $C$ are absolute constants, {and the logarithmic factors are omitted in the big O expression}. A similar result holds for the right singular subspace.
\end{corollary}
For low-rank matrices, Corollary \ref{col:2toinf_full} can be further improved.
\begin{corollary}\label{col:2toinf_low} {When $A$ is of rank-$r$, \eqref{eq:thm} reduces to
\begin{equation}\label{eq:res_low_rank}
\min_{Q\in \mathbb{O}_r}\|\widetilde{U}_1 -U_1Q \|_{2,\infty} \leq c_1\|U_1\|_{2,\infty}\frac{\sigma^2 \bar n}{\sigma_r^2( A)}+c_2\sigma\frac{\sqrt{r}+\sqrt{\log n}}{\sigma_r(A)}.
\end{equation}}
Under the same assumption as in Corollary \ref{col:2toinf_full},
with probability {at least} $1-\frac{c}{n^2}$,
\[
\min_{Q\in \mathbb{O}_r}\|\widetilde{U}_1-U_1Q\|_{2,\infty} \leq C \cdot \frac{\mu_1 \sqrt r+\sqrt{r\log n}+\sqrt r}{\sqrt n} \sim O\left(\sqrt{\frac{ r}{ n}}\right),
\]
where $c$ and $C$ are absolute constants. A similar result holds for the right singular subspace.
\end{corollary}
{
\begin{remark}
    To get a sense of the tightness of this corollary, we mention the following minimax lower bound derived in Theorem \ref{thm:minimax}  \cite{cai2021subspace}. The original result considered the case of missing entries with probability $1-p$, we plug in $p=1$ to reduce it to the fully observed case. \eqref{eq:minimax} below indicates that when $A$ is of rank-$r$, $n\asymp m$, and $\sigma_r\asymp\sigma\sqrt{n}$, the result in Corollary \ref{col:2toinf_low} matches the minimax lower bound $O(\frac{1}{\sqrt{n}})$ up to a factor of $\sqrt r$.
\end{remark}
\begin{theorem}[Theorem 3.3 in \cite{cai2021subspace}]\label{thm:minimax}
    Suppose $1\leq r\leq n/2$, and $(\Delta A)_{ij}\ \stackrel{\mathclap{i.i.d.}}{\sim}\  N(0,\sigma^2)$. Define
    \[
    \mathcal{M}:=\{B\in\mathbb{R}^{n\times m}|\text{rank}(B)=r,\sigma_r(B)\in [0.9\sigma_r^*,1.1\sigma_r^*]\}.
    \]
    Denote by $U(B)\in\mathbb{R}^{n\times r}$ the matrix containing the $r$ left singular vectors of $B$. Then there exists some universal constant $c_{lb}>0$ such that
    \begin{equation}\label{eq:minimax}
        \inf_{\hat U}\sup_{A\in \mathcal{M}}\mathbb{E}\big[\min_{{Q}\in\mathbb{O}_r}\|\hat UQ-U(A)\|_{2,\infty}\big]\geq c_{lb}\min\left\{\frac{\sigma^2}{{\sigma_r^*}^2}\sqrt{nm}+\frac{\sigma}{{\sigma_r^*}}\sqrt{n},1\right\}\frac{1}{\sqrt{n}},
    \end{equation}
    where the infimum is taken over all estimators for $U(A)$ based on the noisy observation $A+\Delta A$.
\end{theorem}}
\begin{table}
\centering
\resizebox{\textwidth}{!}{
\begin{tabular}{c|c|c|c|c}
\hline
\ & \thead{Achieve $O(\sqrt{r/n})$ for\\rank-$r$ matrices} & Do not require $\kappa=O(1)$ &Do not require addition assumptions &\thead{Achieve $O(r/\sqrt{n})$ for\\general matrices} \\ 
\hline 
\cite{cape2019two} & \xmark & \cmark & \cmark &\xmark \\
\cite{lei2019unified} & \cmark & \cmark &\xmark &\xmark\\
\cite{abbe2020entrywise} & \cmark & \xmark & \cmark &\xmark \\
\cite{cai2021subspace} & \cmark & \xmark & \xmark &\xmark \\
This paper & \cmark & \cmark & \cmark &\cmark\\
\hline
\end{tabular}
}
\caption{{Comparison of result in this paper with existing works about the $\ell_{2,\infty}$ bound. Here we compare these results under the assumptions that $\Delta A$ has i.i.d. $N(0,\sigma^2)$ entries, $\|U_1\|_{2,\infty}\leq c\sqrt{r/n}$, where $\sigma,c$ are constants, and $n\asymp m$. Here $\kappa=\frac{\sigma_1(A)}{\sigma_r(A)}$ is the condition number of $A$. For simplicity, we ignore $\log n$ in the big O notation.} }\label{table:datasets}
\label{tab:1}
\end{table}
{We take a moment here to make a comparison between several existing works with the results derived in this paper (Theorem 2.1 and its corollaries). Table \ref{tab:1} summarizes the $\ell_{2,\infty}$-norm bound and requirements in each work. The purpose of the comparison is to show the effectiveness of the derived results under the setting where the perturbation matrix $\Delta A$ has i.i.d. Gaussian entries and the matrix $A$ is nearly square ($n\asymp m$). To be fair, we would like to mention that some of the existing results might be better suited for other settings (such as when $m\gg n$). }

From the table, we can see that the result in this paper achieves the $O({\sqrt{r/n}})$ order upper bound for low-rank matrices and $O(r/\sqrt{n})$ for full-rank matrices. In comparison, previous results in \cite{cape2019two} do not achieve the $O({\sqrt{r/n}})$ order for the low-rank case {under the assumptions as in Corollary \ref{col:2toinf_low}}.  The result in  \cite{lei2019unified} achieves the same order of accuracy $O({\sqrt{r/n}})$ but only for rank-$r$ matrices and under a more restrictive gap condition (below is a simplified version)
\[
\sigma_r-\sigma_{r+1}\geq c\left(\min\{\frac{\sigma_1}{\sigma_r},2r\}\sigma\sqrt{\bar n}+\|A\|_{2,\infty}\right),
\]
where $c$ is a constant.
\cite{abbe2020entrywise,chen2021asymmetry,cheng2020inference,eldridge2018unperturbed} consider the perturbation of one eigen-vector instead of a set of eigen-vectors. {\cite{abbe2020entrywise} obtained an $O({\sqrt{r/n}})$ perturbation bound for low-rank matrices, but for full rank matrices, their bound is $O(1)$. Besides, the bound contains in it the condition number $\widetilde{\kappa}(A)=\frac{\sigma_1(A)}{\sigma_r(A)-\sigma_{r+1}(A)}$, which potentially makes it very large. Likewise, the result in \cite{cai2021subspace} also contains a condition number ${\kappa}(A)=\frac{\sigma_1(A)}{\sigma_r(A)}$. In contrast, the condition number does not show up in Theorem \ref{thm:2toinf_full}, hence it is more suitable for matrices with large condition numbers.} In addition, all previous analyses {except for the one in \cite{cape2019two}} are based on techniques developed for eigen-decomposition. When they are applied to rectangular matrices $A \in \mathbb{R}^{n\times m}$, the matrix needs to be symmetrized, causing the resulting upper bounds potentially depend on both the left and the right singular vectors. In contrast, our bound \eqref{eq:thm} is one-sided, in that the perturbation of $U_1$ only depends on $\|U_1\|_{2,\infty}$ but not $\|V_1\|_{2,\infty}$.

{We make a more detailed comparison between the present Theorem \ref{thm:2toinf_full} and Theorem 4.3 in \cite{cape2019two}. 
For a fair comparison, we restrict ourselves to the setting where both \cite{cape2019two} and our results hold, that is, when the data matrix $A\in\mathbb{R}^{n\times m}$ is of rank $r$  and the perturbation matrix $\Delta A$ has i.i.d. Gaussian  $N(0,\sigma^2)$  entries with some constant $\sigma$. Theorem 4.3 in \cite{cape2019two} reads if $n\geq m,\ \sigma_r(A)\gtrsim\sigma(n/\sqrt m)$, then
    \begin{equation}\label{eq:prevbound}
        \min_{Q\in\mathbb{O}_r}\|\widetilde{U}_1-U_1Q\|_{2,\infty}\leq C_r\left(\frac{\log n}{\sigma_r(A)}\right)\left(1+\frac{m}{\sigma_r(A)}+\frac{\sqrt{m}}{\log n}\|U_1\|_{2,\infty}\right),
    \end{equation}
    where $C(r) \sim O(\sqrt r)$.
    Comparing \eqref{eq:prevbound} and \eqref{eq:res_low_rank}, we notice that under the allowable gap condition $\sigma_r(A)\gtrsim\sigma\sqrt{\max\{n,m\}}$ and the incoherence condition $\|U_1\|_{2,\infty}\leq c_r/\sqrt{n}$, the second term in \eqref{eq:prevbound} can be as large as $O(1)$, which is much larger than our bound $O(\sqrt{r/n})$.
}

Admittedly, our current result only holds for Gaussian perturbation {due to the proof techniques we use}. We leave the study of other perturbation types as future work. 

\subsection{Stability of principal component analysis (PCA)}\label{sec:PCA_stability}
{As one of the arguably most popular tools for data visualization and exploration, PCA is used to extract the main features from a dataset or to reduce the dimensionality of the data \cite{nguyen2019ten}. There is a vast literature on the analysis of PCA. Most previous works focused on the consistency of Principal Component directions or eigenvalues \cite{chen2021spectral,cai2021subspace,vaswani2017finite,narayanamurthy2020fast}, while the stability of PC scores (i.e., the projection of data matrix $A$ onto its PC directions, using the notations in this paper, PC scores are given by $U_1\Sigma_1$) is less explored, despite its importance in the analysis of various spectral methods.} 

{There are several relevant works investigating the stability of PC scores, but their analyses were under different settings. For completeness, we include a brief review here. \cite{abbe2022} developed an $\ell_p$ analysis for a hollowed version of PCA, where SVD is conducted on the hollowed gram matrix $G=\mathcal{H}(AA^T)$. Here, $A$ is the data matrix and the operator $\mathcal{H}(\cdot)$ zeros out all diagonal entries of a square matrix. The PC scores are given by $U\Lambda^{\frac{1}{2}}$, where $U$ and $\Lambda$ are the eigenvector matrix and corresponding eigenvalues of $G$. Perturbation bounds on PC scores in $\ell_{2,p}$ norm were derived in \cite{abbe2022} to characterize entrywise behaviour of PCA. Another line of research studied adjacency spectral embedding (ASE) for random dot product graphs (RDPG), which is closely related to PCA in that they both return a weighted singular vector matrix. Central limit theorems for rows of ASE have been provided in \cite{tang2018limit,athreya2021estimation}.} 

{Different from these previous studies, in this section, we focus on the stability of PC scores of the original PCA algorithm, which does not use the hallowed gram matrix.} Given a centered data matrix $A$ and its conformal SVD, PCA returns $U_1\Sigma_1$ (or $V_1\Sigma_1$) as the low dimensional projection into $\mathbb{R}^r$. Due to the possible similarity among singular values within $\Sigma_1$, the PCA embedding may be subject to rotations. Hence when computing the error, we mode out this rotation and aim to bound $\min_{Q\in \mathbb{O}_r}\|U_1\Sigma_1-\widetilde{U}_1\widetilde{\Sigma}_1Q\|$ or $\min_{Q\in \mathbb{O}_r}\|U_1\Sigma_1-\widetilde{U}_1\widetilde{\Sigma}_1Q\|_F$, where $\|\cdot\|$ is the spectral norm. 

The main difference between these quantities and the sin$\Theta$ angle between singular subspaces is that $U_1$ is now multiplied by the corresponding singular values, and it is the perturbation of this product that we want to analyze. Naively, one may expect that the perturbation of $U_1\Sigma_1$ is approximately equal to the perturbation of $U_1$ times $\|\Sigma_1\|$ plus the perturbation of $\Sigma$ times $\|U_1\|$, and the perturbation of $U_1$ can in turn be controlled by the sin$\Theta$ theorem. This argument leads to
\begin{equation}\label{eq:sigma1}
\min_{Q\in \mathbb{O}_r}\|U_1\Sigma_1-\widetilde{U}_1\widetilde{\Sigma}_1Q\| \leq  c\cdot \frac{\sigma_1(A)\|\Delta A\|}{\sigma_r-\sigma_{r+1}},
\end{equation} where $c$ is some absolute constant. {However, this bound is quite large due to the existence of $\sigma_1(A)$ in the numerator. Noticing that $\sigma_1(A)$ appears in \eqref{eq:sigma1} because we consider $U_1\Sigma_1$ as a whole, in the following theorem, we show that the perturbed singular vectors corresponding to different singular values actually have different levels of stability, which in turn enables a tighter bound on the PC scores. More specifically, the next theorem shows that the singular vectors associated with larger singular values are more stable.}
\begin{theorem}\label{thm:single}
For $j=1,...,r$, let $\sin\Theta(\widetilde{u}_j, U_1)$ be the sin$\Theta$ angle between the $j$th left perturbed singular vector $\widetilde{u}_j$ and the leading $r$-dimensional singular subspace $span(U_1)$ of $A$. Then provided that $3\|\Delta A\| \leq  \sigma_r -\sigma_{r+1}$, we have
 \[
\|\sin\Theta(\widetilde{u}_j, U_1)\| \leq \frac{C\|\Delta A\|}{\sigma_j-\sigma_{r+1}}.
\]
where $C$ is some universal constant and by definition $\|\sin\Theta(\widetilde{u}_j, U_1)\| \equiv \|\widetilde{u}_j^T U_2\|$, $U_2$ is the orthogonal complement of $U_1$.
\end{theorem}
The different levels of stability of singular vectors observed in Theorem \ref{thm:single} will help us get rid of the $\sigma_1(A)$ and establish a tighter bound on the PC scores.
\begin{theorem}\label{thm:pca}
$\widetilde A=A+\Delta A$, $U_1\Sigma_1$ is the PCA embedding of $A$ and $\widetilde{U}_1\widetilde{\Sigma}_1$ is that of $\widetilde{A}$, we have
$$\min_{Q\in \mathbb{O}_r}\|U_1\Sigma_1-\widetilde{U}_1\widetilde{\Sigma}_1Q\|\leq 3\|\Delta A\|+3\sigma_{r+1}\min\left\{\frac{2\|\Delta A\|}{\sigma_r-\sigma_{r+1}},1\right\}, 
$$
\vskip -0.7cm
\begin{align*}
\min_{Q\in \mathbb{O}_r}\|U_1\Sigma_1-\widetilde{U}_1\widetilde{\Sigma}_1Q\|_F&\leq\left(2\|(\Delta A)_r\|_F^2+3\left(\|(\Delta A)_r\|_F+\|(\Sigma_2)_r\|_F\min\left\{\frac{2\|\Delta A\|}{\sigma_r-\sigma_{r+1}},1\right\}\right)^2\right)^{1/2}\\&\quad\ +\|(\Delta A)_r\|_F+\|(\Sigma_2)_r\|_F\min\left\{\frac{2\|\Delta A\|}{\sigma_r-\sigma_{r+1}},1\right\}.
\end{align*}
Here $(\Delta A)_r$ is the best rank-$r$ approximation of $\Delta A$.
\end{theorem}
The upper bound is tighter than \eqref{eq:sigma1} and can be used to facilitate the error analysis of PCA-related methods (e.g., \cite{little2018analysis}).
\begin{remark}
When $A$ is rank-$r$, the above result reduces to
$$\min_{Q\in \mathbb{O}_r}\|U_1\Sigma_1-\widetilde{U}_1\widetilde{\Sigma}_1Q\|\leq 3\|\Delta A\|,$$
$$
\min_{Q\in \mathbb{O}_r}\|U_1\Sigma_1-\widetilde{U}_1\widetilde{\Sigma}_1Q\|_F\leq(\sqrt{5}+1)\|(\Delta A)_r\|_F.
$$
\end{remark}
\begin{remark}
    {A tight $\ell_{2,p}$ norm perturbation bound of PC scores for hollowed PCA was developed in \cite{abbe2022}. However, unlike the previous theorem for vanilla PCA, it seems not possible to eliminate $\sigma_1(A)$ from the bound in hollowed PCA, due to the fact that the hollowed PCA conducts the decomposition on the gram matrix instead of the original data matrix $A$. In the noisy setting, the noise on the Gram matrix contains the term $A^T\Delta A$, whose norm may reach $O(\sigma_1(A)\|\Delta A\|)$ with $\sigma_1$ included in the expression.  } 

\end{remark}

\subsection{A new stability result on singular value {truncation}}\label{sec:SVT_stability}
In addition to studying the perturbations of $U_1$ and $U_1\Sigma_1$, we also investigate the stability of the hard singular value thresholding operator,  which provides the best rank-$r$ approximation of $A$, i.e., $A_r=U_1\Sigma_1V_1^T$. This operator, also known as singular value truncation, is widely used in matrix completion and matrix denoising for promoting low-rankness or reducing the noise \cite{tanner2013normalized,donoho2014minimax,cai2010singular,gavish2014optimal}. Let $\widetilde{A}=A+\Delta A$ be the noisy matrix, and let $\widetilde{A}_r$ denote its best rank-$r$ approximation. We characterize the stability of the hard singular value thresholding operator through a bound on $\|A_r-\widetilde{A}_r\|$.
{Previous works have investigated the stability of truncated SVD \cite{luo2021schatten,vu2021perturbation}, and tight error bounds for low-rank matrices have been derived. However, a tight bound for general matrices is still missing in the literature.} 

For rank-$r$ matrix $A$ with $r<\min\{m,n\}$, the following perturbation result was obtained in \cite{luo2021schatten},
\begin{equation}\label{eq:thr}
\|A-\widetilde{A}_r\| \leq 2\|\Delta A\|.
\end{equation}
Since in practice, $A$ may not be exactly rank-$r$, we hope to establish upper bounds for general full-rank matrices. 

We comment that although we can easily derive an upper bound of full-rank matrices from that of the low-rank ones, the resulting bound is not tight. Explicitly, for a full-rank matrix $A$, $A_r$ is of low rank, so we can apply \eqref{eq:thr} on $A_r$ to get
\begin{align*}
\|A_r- \widetilde{A}_r\|&= \|A_r - (A+\Delta A)_r\| = \|A_r - (A_r+\widetilde{E})_r\|
\leq 2\|\widetilde{E}\| \leq 2\|\Delta A\|+2\|A-A_r\| =2\|\Delta A\|+2\sigma_{r+1},
\end{align*}
where $\widetilde{E}=\Delta A+A-A_r$ and the first inequality used \eqref{eq:thr}. 

Apparently, this bound is not optimal as it does not shrink to 0 when $\Delta A\rightarrow 0$. This then motivates us to establish the following tighter bound.

\begin{theorem}[Perturbation result on singular value truncation]\label{thm:thresh} Let $A \in \mathbb{R}^{n\times m}$ be any $n\times m$ matrix and $\widetilde{A}=A+\Delta A$ be its noisy version. Denote by $A_r$ and $\widetilde A_r$ their rank-$r$ thresholding with all but the first $r$ singular values set to $0$. Let $\sigma_i$ be the $i$th largest singular value of $A$ and $\Sigma_2$ be the diagonal matrix containing $\sigma_{r+1},...,\sigma_{\min}$ (the $(r+1)$'th to the last singular values of $A$) on the diagonal. Then
\begin{equation}\label{eq:2norm_threholding}
    \|A_r-\widetilde A_r\|\leq 2\|\Delta A\|+2\sigma_{r+1}\min\left\{\frac{2\|\Delta A\|}{\sigma_r-\sigma_{r+1}},1\right\},
\end{equation}
\begin{align}\label{eq:Fnorm_threholding}
    \|A_r-\widetilde A_r\|_F&\leq \left(2\|(\Delta A)_r\|_F^2+3\left(\|(\Delta A)_r\|_F+\|(\Sigma_2)_r\|_F\min\left\{\frac{2\|\Delta A\|}{\sigma_r-\sigma_{r+1}},1\right\}\right)^2\right)^{1/2}.
\end{align}
\end{theorem}
This error bound has exactly the same form as the PCA perturbation bound established in the previous section, except that here $A_r$ and $\widetilde A_r$ do not differ by a rotation. Intuitively, this indicates that the noise-induced rotation on $\widetilde U_1$ and that on $\widetilde V_1^T$ can essentially cancel with each other.


\begin{remark}
When $A$ is a rank-$r$ matrix, the bound in Theorem \ref{thm:thresh} reduces to the result in \cite{luo2021schatten}:
\begin{equation}
\begin{cases} \|A_r-\widetilde{A}_r\|\leq 2\|\Delta A\|,\\
 \|A_r-\widetilde{A}_r\|_F\leq \sqrt{5}\|(\Delta A)_r\|_F.
\end{cases}
\end{equation} 
\end{remark}
\section{Closed-form expression of $\sin\Theta$ distance between two singular spaces}\label{sec:3}
The several new results presented in the previous section are derived either directly or indirectly from a set of sin$\Theta$ formulae we shall establish in this section. In other words, these sin$\Theta$ formulae serve as useful tools to analyze SVD based perturbation problems. 
\subsection{First order equivalent expressions of the $\sin\Theta$ distance}

Following the same notation as in Section \ref{sec:2.1}, our goal is to compute the exact expressions of perturbation angles of the leading left singular subspace $U_1$ under noise $\Delta A$.

For two matrices $U_1,\widetilde U_1 \in \mathbb{R}^{n\times r}$ with orthonormal columns, let the singular values of $U_1^T\widetilde U_1$ be $\gamma_1\geq\gamma_2\geq...\geq \gamma_r\geq 0$, then $\cos^{-1}{\gamma_i}$, $i=1,...,r$ are the principal angles, and the $\sin\Theta$ matrix is the following diagonal matrix 
\[
\sin \Theta(U_1,\widetilde U_1)=\text{diag}\{\sin \cos^{-1}(\gamma_1),\sin \cos^{-1}(\gamma_2),...,\sin\cos^{-1}(\gamma_r)\}.
\]
 The angles are usually measured  under either the spectral norm $\|\sin\Theta(U_1,\widetilde U_1)\|$ or the Frobenius norm $\|\sin\Theta(U_1,\widetilde U_1)\|_F$. It is well known that (e.g., \cite{cai2018rate,knyazev2002principal})
\begin{equation}\label{eq:sinform}
\|\sin\Theta(U_1,\widetilde U_1)\|=\|U_2^T\widetilde U_1\|=\|\widetilde U_2^T U_1\|,
\end{equation}
\begin{equation}\label{eq:sinform1}
\|\sin\Theta(U_1,\widetilde U_1)\|_F=\|U_2^T\widetilde U_1\|_F=\|\widetilde U_2^T U_1\|_F,
\end{equation}
where $U_2$ is the orthogonal complement of $U_1$ as defined in \eqref{eq:conformal}.

\eqref{eq:sinform} and \eqref{eq:sinform1} indicate that the matrices $U_2^T\widetilde U_1$ and $\widetilde U_2^T U_1$ are  key intermediate quantities to bound the sin$\Theta$ angles. In the following theorem, we provide useful expressions of these key quantities.

\begin{theorem}[Angular perturbation formula]\label{thm:main}
Let $A$, $\widetilde{A}=A+\Delta A$ be two $n\times m$ matrices and their conformal SVDs are defined as \eqref{eq:conformal}. The rank of $A$ is at least $r$. Assume there is a gap between the $r$th and the $(r+1)$th singular values, i.e., $\sigma_r-\widetilde\sigma_{r+1}>0$ and  $\widetilde\sigma_r-\sigma_{r+1}>0$.  Then the following expressions hold:
\begin{equation}
\begin{aligned}\label{eq:main}
    U_1^T\widetilde U_2&=F_U^{12}\circ(U_1^T(\Delta A)\widetilde V_2\widetilde\Sigma_2^T+\Sigma_1 V_1^T(\Delta A)^T\widetilde U_2),\\
    U_2^T\widetilde U_1&=F_U^{21}\circ(U_2^T(\Delta A)\widetilde V_1\widetilde\Sigma_1^T+\Sigma_2 V_2^T(\Delta A)^T\widetilde U_1),\\
    V_1^T\widetilde V_2&=F_V^{12}\circ(\Sigma_1^T U_1^T(\Delta A)\widetilde V_2+ V_1^T(\Delta A)^T\widetilde U_2\widetilde\Sigma_2),\\
    V_2^T\widetilde V_1&=F_V^{21}\circ(\Sigma_2^T U_2^T(\Delta A)\widetilde V_1+ V_2^T(\Delta A)^T\widetilde U_1\widetilde\Sigma_1).
\end{aligned}
\end{equation}
More specifically, the assumption $\sigma_r-\widetilde\sigma_{r+1}>0$ is required for the first and the third expressions of \eqref{eq:main} to hold, and $\widetilde\sigma_r-\sigma_{r+1}>0$ is required for the second and the last expressions to hold. Here $\circ$ means the Hadamard product, or element-wise product between two matrices. $F_U^{12}\in\mathbb{R}^{r\times (n-r)}$ has entries $(F_U^{12})_{i,j}=\frac{1}{\widetilde\sigma_{j+r}^2-\sigma_i^2},\ 1\leq i\leq r,\ 1\leq j\leq n-r$; $F_U^{21}\in\mathbb{R}^{(n-r)\times r}$ has entries $(F_U^{21})_{i,j}=\frac{1}{\widetilde\sigma_j^2-\sigma_{i+r}^2},\ 1\leq i\leq n-r,\ 1\leq j\leq r$. Similarly, $F_V^{12}\in\mathbb{R}^{r\times (m-r)}$ has entries $(F_V^{12})_{i,j}=\frac{1}{\widetilde\sigma_{j+r}^2-\sigma_i^2},\ 1\leq i\leq r,\ 1\leq j\leq m-r$; and $F_V^{21}\in\mathbb{R}^{(m-r)\times r}$ with entries $(F_V^{21})_{i,j}=\frac{1}{\widetilde\sigma_j^2-\sigma_{i+r}^2},\ 1\leq i\leq m-r,\ 1\leq j\leq r$. Here if $i>\min\{n,m\}$, we enforce $\sigma_i$ and $\widetilde\sigma_i$ to be 0.
\end{theorem}

Taking the spectral norm on both hand sides of \eqref{eq:main} gives us the following new expression of the sin$\Theta$ distance.
\begin{corollary}\label{col:main}
If the condition in Theorem \ref{thm:main} is satisfied, then the $\sin\Theta$ distances between the $r$ leading singular spaces of the original and the perturbed matrices satisfy
\begin{align*}
\|\sin\Theta(U_1,\widetilde U_1)\|&=\|F_U^{12}\circ(U_1^T(\Delta A)\widetilde V_2\widetilde\Sigma_2^T+\Sigma_1 V_1^T(\Delta A)^T\widetilde U_2) \|\\ &=\|F_U^{21}\circ(U_2^T(\Delta A)\widetilde V_1\widetilde\Sigma_1^T+\Sigma_2 V_2^T(\Delta A)^T\widetilde U_1)\|,\\
\|\sin\Theta(V_1,\widetilde V_1)\|&=\|F_V^{12}\circ(\Sigma_1^T U_1^T(\Delta A)\widetilde V_2+ V_1^T(\Delta A)^T\widetilde U_2\widetilde\Sigma_2) \|\\&=\|F_V^{21}\circ(\Sigma_2^T U_2^T(\Delta A)\widetilde V_1+ V_2^T(\Delta A)^T\widetilde U_1\widetilde\Sigma_1)\|.
\end{align*}
\end{corollary}
\begin{remark} 
In the expressions of corollary \ref{col:main},  the singular value gaps are contained in the terms $F_U^{21}$, $F_U^{12}$ $F_V^{21}$, $F_V^{12}$ as denominators. In this sense, \eqref{eq:main} conveys the same insight as Wedin's sin$\Theta$ theorem.  
\end{remark}
Everything else in the right hand sides of Corollary \ref{col:main} is straightforward to bound except perhaps for the Hadamard products. The following lemma shows that the Hadamard product is also relatively easy to treat.

\begin{lemma}\label{lemma:Hnorm}
Assume $\sigma_r-\widetilde\sigma_{r+1}>0,\ \widetilde\sigma_r-\sigma_{r+1}>0$, let $F_U^{12},\ F_U^{21}$, $\Sigma_1,\ \widetilde\Sigma_1$, $\Sigma_2,\ \widetilde\Sigma_2$, be the same as in Theorem \ref{thm:main} and let $H_1\in\mathbb{R}^{(n-r)\times r},\ H_2\in\mathbb{R}^{(m-r)\times r},\ H_3\in\mathbb{R}^{r\times (m-r)},\ H_4\in\mathbb{R}^{r\times (n-r)}$ be some arbitrary matrices. Then
\begin{equation}\label{eq:Hnorm_ineq}
|||F_U^{21}\circ (H_1\widetilde\Sigma_1) ||| \leq\frac{\widetilde\sigma_r}{\widetilde\sigma_r^2-\sigma_{r+1}^2}|||H_1|||,\ ||| F_U^{21}\circ(\Sigma_2 H_2)||| \leq\frac{\sigma_{r+1}}{\widetilde\sigma_r^2-\sigma_{r+1}^2}|||H_2|||,
\end{equation}
\begin{equation}\label{eq:Hnorm_ineq2}
|||F_U^{12}\circ(H_3\widetilde\Sigma_2^T)|||\leq\frac{\widetilde\sigma_{r+1}}{\sigma_r^2-\widetilde\sigma_{r+1}^2}|||H_3|||,\ |||F_U^{12}\circ(\Sigma_1 H_4)|||\leq\frac{\sigma_{r}}{\sigma_r^2-\widetilde\sigma_{r+1}^2}|||H_4|||,
\end{equation}
where $|||\cdot |||$ can be either the spectral or the Frobenius norm. Similar results also hold for $F_V^{12}$ and $F_V^{21}$.
\end{lemma}
\subsection{Examples in using Theorem \ref{thm:main}}
We demonstrate how to use Theorem \ref{thm:main} to simplify proofs of some existing perturbation bounds in the literature. The theorem we use to derive all the new results in this paper is in the next section (Theorem \ref{thm:main2}). Curious readers may safely jump to the next section from here.  \\
\textbf{Example 1:} The angular perturbation formulae in Theorem \ref{thm:main} naturally yield the one-sided sin$\Theta$ bounds first discovered in \cite{cai2018rate}. Theorem \ref{thm:main} now introduces a very straightforward derivation of these bounds.
\begin{theorem}[One-sided sin$\Theta$ theorem]\label{thm:two-sided} Using the same notation and quantities as in Theorem \ref{thm:main}, if $\sigma_r-\widetilde\sigma_{r+1}>0,\ \widetilde\sigma_r-\sigma_{r+1}>0$, then
\small
\begin{align}
\|\sin\Theta(U_1,\widetilde U_1)\|& \leq \min\left\{\frac{\widetilde\sigma_r\|(\Delta A)\widetilde V_1\|}{\widetilde\sigma_r^2-\sigma_{r+1}^2}+ \frac{\sigma_{r+1}\|(\Delta A)^T\widetilde U_1\|}{\widetilde\sigma_r^2-\sigma_{r+1}^2},\frac{\sigma_r\|(\Delta A)V_1\|}{\sigma_r^2-\widetilde\sigma_{r+1}^2}+ \frac{\widetilde\sigma_{r+1}\|U_1^T(\Delta A) \|}{\sigma_r^2-\widetilde\sigma_{r+1}^2}\right\}, \label{eq:onesideU}\end{align}
\begin{equation}
\|\sin\Theta(V_1,\widetilde V_1)\|\leq \min\left\{\frac{\widetilde\sigma_r\|\widetilde U_1^T(\Delta A)\|}{\widetilde\sigma_r^2-\sigma_{r+1}^2}+ \frac{\sigma_{r+1}\|(\Delta A)\widetilde V_1\|}{\widetilde\sigma_r^2-\sigma_{r+1}^2},\frac{\sigma_r\|U_1^T(\Delta A)\|}{\sigma_r^2-\widetilde\sigma_{r+1}^2}+ \frac{\widetilde\sigma_{r+1}\|(\Delta A)V_1 \|}{\sigma_r^2-\widetilde\sigma_{r+1}^2}\right\}.\label{eq:onesideV}
\normalfont
\end{equation}
Moreover,
\begin{equation}
\max\{\|\sin\Theta(U_1,\widetilde U_1)\|,\|\sin\Theta(V_1,\widetilde V_1)\|\} \leq \min\left\{\frac{1}{\sigma_r-\widetilde\sigma_{r+1}},\frac{1}{\widetilde\sigma_r-\sigma_{r+1}}\right\}\|\Delta A\|.
\label{eq:uniform_hat}
\end{equation}
\end{theorem}
\eqref{eq:onesideU} and \eqref{eq:onesideV} are individual bounds on $\|\sin\Theta(U_1,\widetilde U_1)\|$ and $\|\sin\Theta(V_1,\widetilde V_1)\|$, while the classical Wedin's sin$\Theta$ theorem is a uniform bound on both $\|\sin\Theta(U_1,\widetilde U_1)\|$ and $\|\sin\Theta(V_1,\widetilde V_1)\|$. 
The benefit of obtaining the individual bounds was clearly pointed out in  \cite{cai2018rate} by an example. When $A\in \mathbb{R}^{n\times m}$ is a fixed rank-$r$
matrix with $r < n \ll m$, and $\Delta A\in \mathbb{R}^{n\times m}$ is a small random matrix with i.i.d. standard normal entries.  The Wedin's theorem implies 
\begin{equation}\label{eq:uniform_ori}
\max\{\|\sin\Theta(U_1,\widetilde U_1)\|,\|\sin\Theta(V_1,\widetilde V_1)\|\}\leq\frac{C\max\{\sqrt{n}, \sqrt{m}\}}{\sigma_r},
\end{equation}
while the one-sided bounds approximately give,
\begin{equation}\label{eq:two-side}
\|\sin\Theta(U_1,\widetilde U_1)\| \leq \frac{C\sqrt{n}}{\sigma_r} , \quad \|\sin\Theta(V_1,\widetilde V_1)\|\}\leq\frac{C\sqrt{m} }{\sigma_r}.
\end{equation}
Since we assumed $n\ll m$, only the one-sided bound successfully indicated that $U_1$ is more stable than $V_1$. 

 The proof of Theorem \ref{thm:two-sided} is a simple application of Theorem \ref{thm:main}.
\begin{proof} From Theorem \ref{thm:main} we have 
\[
U_2^T\widetilde U_1=F_U^{21}\circ(U_2^T(\Delta A)\widetilde V_1\widetilde\Sigma_1^T+\Sigma_2 V_2^T(\Delta A)^T\widetilde U_1),
\]
\[
 U_1^T\widetilde U_2=F_U^{12}\circ(U_1^T(\Delta A)\widetilde V_2\widetilde\Sigma_2^T+\Sigma_1 V_1^T(\Delta A)^T\widetilde U_2).
\]
By \eqref{eq:Hnorm_ineq} in Lemma \ref{lemma:Hnorm}, 
\begin{align}
    \|U_2^T\widetilde U _1\|&\leq\frac{\widetilde\sigma_r}{\widetilde\sigma_r^2-\sigma_{r+1}^2}\|U_2^T(\Delta A)\widetilde V_1\|+ \frac{\sigma_{r+1}}{\widetilde\sigma_r^2-\sigma_{r+1}^2}\|V_2^T(\Delta A)^T\widetilde U_1\|\notag\\
   & \leq\frac{\widetilde\sigma_r}{\widetilde\sigma_r^2-\sigma_{r+1}^2}\|(\Delta A)\widetilde V_1\|+ \frac{\sigma_{r+1}}{\widetilde\sigma_r^2-\sigma_{r+1}^2}\|(\Delta A)^T\widetilde U_1\| \label{eq:tight_1side}\\
   & \leq\frac{\|\Delta A\|}{\widetilde{\sigma}_r-\sigma_{r+1}} \label{eq:relax_1side}.
\end{align}
Similarly,
\begin{align}
\|U_1^T\widetilde U_2\|& \leq  \frac{\widetilde \sigma_{r+1}}{\sigma_r^2-\widetilde\sigma_{r+1}^2}\|U_1^T\Delta A\|+ \frac{\sigma_{r}}{\sigma_r^2-\widetilde\sigma_{r+1}^2}\|(\Delta A)V_1 \| \label{eq:U1U2}\\
&  \leq \frac{\|\Delta A\|}{\sigma_r-\widetilde\sigma_{r+1}} \label{eq:relax_1side2}.
\end{align}
Inserting \eqref{eq:tight_1side} and \eqref{eq:U1U2} into $\|\sin\Theta(U_1,\widetilde U_1)\| = \min\{\|U_1^T\widetilde U_2\|,\|U_2^T\widetilde U _1\| \}$, we obtain \eqref{eq:onesideU}. Similarly, \eqref{eq:onesideV} also holds. \eqref{eq:uniform_hat} is obtained by using \eqref{eq:relax_1side} and \eqref{eq:relax_1side2}. 
\end{proof}

\noindent\textbf{Example 2:} In this example, we show that one may obtain some interesting results when applying Theorem \ref{thm:main} to some less usual choices of $\Delta A$. 

Explicitly, we use Theorem \ref{thm:main} to re-derive a useful result in \cite{cai2018rate} but with a more straightforward proof. The result, copied in Proposition \ref{pro:1}, is about the sin$\Theta$ distance between the leading singular subspace of a matrix $A$ and an arbitrary subspace.


\begin{proposition}[Proposition 1 in \cite{cai2018rate}]\label{pro:1}
Suppose $A\in\mathbb{R}^{n\times m}$. The orthonormal matrix $V=[V_1,V_2]\in\mathbb{R}^{m\times m}$ is the matrix of right singular vectors of $A$, i.e., $V_1\in\mathbb{R}^{m\times r},\ V_2\in\mathbb{R}^{m\times (m-r)}$ correspond to the first $r$ and last $m-r$ singular vectors respectively. $[W_1,W_2]\in\mathbb{R}^{m\times m}$ is any orthonormal matrix with $W_1\in\mathbb{R}^{m\times r},\ W_2\in\mathbb{R}^{m\times (m-r)}$. Given that $\sigma_r(AW_1)>\sigma_{r+1}(A)$, we have
\begin{equation}\label{eq:prop1a}
    \|\sin\Theta(V_1,W_1)\|\leq\min\left\{\frac{\sigma_r(AW_1)\|\mathbb{P}_{(AW_1)}AW_2\|}{\sigma_r^2(AW_1)-\sigma_{r+1}^2(A)}, 1\right\}.
\end{equation}
\begin{equation}\label{eq:prop1b}
    \|\sin\Theta(V_1,W_1)\|_F\leq\min\left\{\frac{\sigma_r(AW_1)\|\mathbb{P}_{(AW_1)}AW_2\|_F}{\sigma_r^2(AW_1)-\sigma_{r+1}^2(A)}, \sqrt{r}\right\}.
\end{equation}
\end{proposition}
In order to use Theorem \ref{thm:main} to prove Proposition \ref{pro:1}, we recognize that Proposition \ref{pro:1} is actually a sin$\Theta$ bound under a special perturbation. Specifically, if we set 
$\Delta A=AW_1 W_1^T-A$, then the quantity $\sin\Theta(V_1,W_1)$  bounded in Proposition \ref{pro:1} is exactly the sin$\Theta$ angle between $A$ and $\widetilde{A}=A+\Delta A$. In addition, this particular choice of $\Delta A$ has small magnitude of norm therefore leading to a small perturbation bound. 
\begin{proof}
Apply Theorem \ref{thm:main} to $A$ and $\widetilde A=AW_1 W_1^T$, which means $\Delta A=\widetilde A -A =AW_1 W_1^T-A=-AW_2 W_2^T$. Assume $U_i, V_i,\Sigma_i,\widetilde{U}_i, \widetilde{V}_i,\widetilde{\Sigma}_i$, $i=1,2$ are from the conformal SVDs \eqref{eq:conformal} of this $A$ and $\widetilde{A}$. Then using the notation in Theorem \ref{thm:main}, we have $(F_V^{21})_{i,j}=\frac{1}{\sigma_j^2(AW_1)-\sigma_{i+r}^2(A)}$,  $\widetilde V_1=W_1,\ \Sigma_2^TU_2^T(\Delta A)\widetilde V_1=0$. Theorem \ref{thm:main} in this case gives
\begin{align*}
    V_2^T W_1&=F_V^{21}\circ (V_2^T(\Delta A)^T\widetilde U_1\widetilde \Sigma_1).
\end{align*}
By Lemma \ref{lemma:Hnorm}, this implies
\[
|||V_2^T W_1|||\leq\frac{\sigma_r(AW_1)|||\widetilde U_1^TAW_2W_2^TV_2|||}{\sigma_r^2(AW_1)-\sigma_{r+1}^2(A)}\leq\frac{\sigma_r(A W_1)|||P_{AW_1} AW_2|||}{\sigma^2_r(AW_1)-\sigma_{r+1}^2(A)},
\]
where $|||\cdot |||$ can be either the spectral of Frobenius norm.
Also, we directly have $\|V_2^T W_1\|\leq 1$ and $\|V_2^T W_1\|_F\leq \sqrt{r}$, thus \eqref{eq:prop1a} and \eqref{eq:prop1b} hold.
\end{proof}

\subsection{High order sin$\Theta$ distance formulae using series expansions}\label{sec:3.3}
Although the formulae in Theorem \ref{thm:main} are already quite useful, they are still only first-order formulae in the following sense. Looking at the first formula in \eqref{eq:main} of Theorem \ref{thm:main}, a closer examination shows that the unknown left hand side  $U_1^T\widetilde{U}_2$ also appears implicitly in the right hand side, albeit as high order terms. Since we consider upper bounds in the non-asymptotic regime, high order errors may sometimes affect the tightness of the bound, so we hope to get rid of them. 

To be more specific about the implicit appearances of the high order terms, we denote the left hand sides of the four formulae in Theorem \ref{thm:main} as $X,Y,W,Z$
\[
X \coloneqq  U_1^T\widetilde{U}_2,  \quad Y\coloneqq  U_2^T \widetilde{U}_1,\quad  W\coloneqq V_1^T\widetilde{V}_2, \quad Z\coloneqq V_2^T \widetilde{V}_1.
\]
First focus on the expression of $Y$ in Theorem \ref{thm:main} 
\begin{align}
Y&\equiv U_2^T\widetilde U_1=F_U^{21}\circ(U_2^T(\Delta A)\widetilde V_1\widetilde\Sigma_1^T+\Sigma_2 V_2^T(\Delta A)^T\widetilde U_1)\notag\\
&=F_U^{21}\circ(\Sigma_2V_2^T(\Delta A)^T U_1U_1^T\widetilde U_1+\Sigma_2 V_2^T(\Delta A)^T U_2U_2^T\widetilde U_1+U_2^T(\Delta A) V_1 V_1^T\widetilde V_1\widetilde\Sigma_1^T\notag\\
&\quad+U_2^T(\Delta A) V_2V_2^T\widetilde V_1\widetilde\Sigma_1^T)\notag\\
&=\underbrace{F_U^{21}\circ(\Sigma_2\alpha_{12}^T U_1^T\widetilde U_1+\alpha_{21}V_1^T\widetilde V_1\widetilde\Sigma_1^T)}_{\coloneqq C_1}+F_U^{21}\circ(\Sigma_2\alpha_{22}^TY)+F_U^{21}\circ(\alpha_{22}Z\widetilde\Sigma_1^T),\label{eq:Y}
\end{align}
where the second line used $U_1U_1^T+U_2U_2^T = I$ and $V_1V_1^T+V_2V_2^T = I$, the third line is a re-grouping of terms, and $\alpha_{ij} := U^T_i\Delta AV_j, 1\leq i,j\leq 2$. 
We can get the same expression for $Z$ 
\begin{align}
    Z&\equiv V_2^T\widetilde V_1=F_V^{21}\circ(\Sigma_2^T U_2^T(\Delta A)\widetilde V_1+ V_2^T(\Delta A)^T\widetilde U_1\widetilde\Sigma_1) \notag\\
    &=F_V^{21}\circ(V_2^T(\Delta A)^T U_1U_1^T\widetilde U_1\widetilde\Sigma_1+ V_2^T(\Delta A)^T U_2U_2^T\widetilde U_1\widetilde\Sigma_1+\Sigma_2^T U_2^T(\Delta A) V_1 V_1^T\widetilde V_1\notag\\
    &\quad+\Sigma_2^T U_2^T(\Delta A) V_2V_2^T\widetilde V_1)\notag\\
    &=\underbrace{F_V^{21}\circ(\alpha_{12}^TU_1^T\widetilde U_1\widetilde\Sigma_1+\Sigma_2^T\alpha_{21}V_1^T\widetilde V_1)}_{\coloneqq C_2}+F_V^{21}\circ(\alpha_{22}^TY\widetilde\Sigma_1)+F_V^{21}\circ(\Sigma_2^T\alpha_{22}Z). \label{eq:Z}
\end{align}
Looking at the last right hand sides of \eqref{eq:Y} and \eqref{eq:Z}, we see that  $Y$ and $Z$ are contained in the second and third terms, respectively, so they appear on both hand sides.

To highlight this structure, we shorten the notation by letting $\mathcal{F}$ be the linear operator defined as
\[
\mathcal{F}\left(\left[\begin{matrix}Y \\ Z\end{matrix}\right]\right) = \left(\begin{matrix}F_U^{21}\circ(\Sigma_2\alpha_{22}^TY)+F_U^{21}\circ(\alpha_{22}Z\widetilde\Sigma_1^T)\\F_V^{21}\circ(\alpha_{22}^TY\widetilde\Sigma_1)+F_V^{21}\circ(\Sigma_2^T\alpha_{22}Z)\end{matrix}\right).
\]
Then \eqref{eq:Y} and \eqref{eq:Z} become,
\[
\left[\begin{matrix}Y \\ Z\end{matrix}\right] = \left(\begin{matrix}C_1\\C_2\end{matrix}\right)+\mathcal{F}\left(\left[\begin{matrix}Y \\ Z\end{matrix}\right]\right) .
\]
Clearly, this is an implicit equation system of $Y$ and $Z$.

Provided $\|\mathcal{F}\|<1$, we can move $\mathcal{F}$ to the left and take the inverse    
\[
\left(\begin{matrix}U_2^T\widetilde{U}_1\\V_2^T\widetilde{V}_1\end{matrix}\right) \equiv \left[\begin{matrix}Y \\ Z\end{matrix}\right] = (1-\mathcal{F})^{-1}\left(\begin{matrix}C_1\\C_2\end{matrix}\right) = \sum_{k=0}^{\infty} \mathcal{F}^k \left(\begin{matrix}C_1\\C_2\end{matrix}\right).
\]
This gives us a series expression of the quantities ${U}_2^T\widetilde{U}_1$ and $V_2^T\widetilde{V}_1$, which allows us to derive Theorem \ref{thm:2toinf_full} and Theorem \ref{thm:single} presented in Section \ref{sec:2}. We summarize this result in the following theorem.

\begin{theorem}[Angular perturbation formula using series expansion]\label{thm:main2}
Using the same notation and quantities as in Theorem \ref{thm:main}, we have 
\begin{equation}\label{eq:implicit}
\left(\begin{matrix}U_2^T\widetilde{U}_1\\V_2^T\widetilde{V}_1\end{matrix}\right)= \left(\begin{matrix}C_1\\C_2\end{matrix}\right)+\mathcal{F}\left(\left[\begin{matrix}U_2^T\widetilde{U}_1\\V_2^T\widetilde{V}_1\end{matrix}\right]\right),  \quad \left(\begin{matrix}U_1^T\widetilde{U}_2\\V_1^T\widetilde{V}_2\end{matrix}\right)= \left(\begin{matrix}C_3\\C_4\end{matrix}\right)+\mathcal{G}\left(\left[\begin{matrix}U_1^T\widetilde{U}_2\\V_1^T\widetilde{V}_2\end{matrix}\right]\right) .
\end{equation}
In addition, provided that $\|\mathcal{F}\|<1$ and $\|\mathcal{G}\|<1$, we have
 \begin{equation}\label{eq:type2}
\left(\begin{matrix}U_2^T\widetilde{U}_1\\V_2^T\widetilde{V}_1\end{matrix}\right) = \sum_{k=0}^{\infty} \mathcal{F}^k \left(\begin{matrix}C_1\\C_2\end{matrix}\right), \quad  \left(\begin{matrix}U_1^T\widetilde{U}_2\\V_1^T\widetilde{V}_2\end{matrix}\right) = \sum_{k=0}^{\infty} \mathcal{G}^k \left(\begin{matrix}C_3\\C_4\end{matrix}\right).
\end{equation}
 Here
\[
\mathcal{F}\left(\begin{matrix}C_1\\C_2\end{matrix}\right) = \left(\begin{matrix}F_U^{21}\circ(\Sigma_2\alpha_{22}^TC_1)+F_U^{21}\circ(\alpha_{22}C_2\widetilde\Sigma_1^T)\\F_V^{21}\circ(\alpha_{22}^TC_1\widetilde\Sigma_1)+F_V^{21}\circ(\Sigma_2^T\alpha_{22}C_2)\end{matrix}\right),\]
\[ \mathcal{G}\left(\begin{matrix}C_3\\C_4\end{matrix}\right) = \left(\begin{matrix}F_U^{12}\circ(\alpha_{11}C_4\widetilde\Sigma_2^T)+F_U^{12}\circ(\Sigma_1\alpha_{11}^T C_3)\\F_V^{12}\circ(\Sigma_1^T\alpha_{11}C_4)+F_V^{12}\circ(\alpha_{11}^T C_3\widetilde\Sigma_2)\end{matrix}\right).
\]
\begin{align*}
C_1&=F_U^{21}\circ(\Sigma_2\alpha_{12}^TU_1^T\widetilde U_1+\alpha_{21}V_1^T\widetilde V_1\widetilde\Sigma_1^T), \quad  
C_2= F_V^{21}\circ(\alpha_{12}^TU_1^T\widetilde U_1\widetilde\Sigma_1+\Sigma_2^T\alpha_{21}V_1^T\widetilde V_1), \\
C_3&=F_U^{12}\circ(\alpha_{12}V_2^T\widetilde V_2\widetilde\Sigma_2^T+\Sigma_1\alpha_{21}^TU_2^T\widetilde U_2) ,\quad
C_4= F_V^{12}\circ(\Sigma_1^T\alpha_{12}V_2^T\widetilde V_2+\alpha_{21}^T U_2^T\widetilde U_2\widetilde \Sigma_2).
\end{align*}
and $\alpha_{ij} := U_i^T\Delta AV_j$. 
\end{theorem}
\begin{remark}
Careful readers may observe that, although we removed all cross terms $U_1^T\widetilde{U}_2$, $U_2^T\widetilde{U}_1$, $V_1^T\widetilde{V}_2$, $V_2^T\widetilde{V}_1$ from the right hand sides of the expressions \eqref{eq:type2}, there are still terms like $U_1^T\widetilde{U}_1$ and $V_1^T\widetilde{V}_1$ appearing on the right hand side. In fact, these terms are of order $O(1)$ thus will not degrade the tightness of the upper bounds by any order of magnitudes and only possibly affect the constants.
\end{remark}
When $A$ has rank $r$, Theorem \ref{thm:main2} reduces to the following simpler formulae. 
\begin{corollary}\label{lemma:Z}
Using the definitions above, when $A$ has rank $r$ and $\|\Delta A\|< \sigma_r(\widetilde A)$, 
\begin{align}\label{eq:cor_low_rank}
\begin{split}
U_2^T\widetilde U_1&=\sum_{k=0}^{+\infty}(\alpha_{22}\alpha_{22}^T)^k(\alpha_{21} V_1^T\widetilde V_1+\alpha_{22}\alpha_{12}^T U_1^T\widetilde U_1\widetilde\Sigma_1^{-1})\widetilde\Sigma_1^{-(2k+1)},\\
V_2^T\widetilde V_1&=\sum_{k=0}^{+\infty}(\alpha_{22}^T\alpha_{22})^k(\alpha_{12}^T U_1^T\widetilde U_1+\alpha_{22}^T\alpha_{21}V_1^T\widetilde V_1\widetilde\Sigma_1^{-1})\widetilde\Sigma_1^{-(2k+1)}.
\end{split}
\end{align}
\end{corollary}
{
\begin{remark}
When $A$ has rank-$r$ and $\alpha_{22}$ is full rank, Corollary \ref{lemma:Z} can also be derived using series expansion for Sylvester-type equations. Denote matrix $M=\begin{pmatrix}
\ & \alpha_{22}\\
\alpha_{22}^T
\end{pmatrix}^{-1}$, $X=\begin{pmatrix}
\ & \alpha_{22}\\
\alpha_{22}^T
\end{pmatrix}\begin{pmatrix}
U_2^T\widetilde U_1\\
V_2^T\widetilde V_1
\end{pmatrix}$, $B=\widetilde\Sigma_1^{-1} $, and $Y=\begin{pmatrix}
\alpha_{21}V_1^T\widetilde V_1 \widetilde\Sigma_1^{-1}\\
\alpha_{12}^TU_1^T\widetilde U_1 \widetilde\Sigma_1^{-1}
\end{pmatrix}$. Direct calculation gives $MX-XB=Y$. By the assumption in Corollary \ref{lemma:Z}, $\|\alpha_{22}\|\leq\|\Delta A\|<\sigma_r(\widetilde A)$, we can see for any eigenvalue $\lambda$ of matrix $M$, it holds that $|\lambda|>\frac{1}{\sigma_r(\widetilde A)}$. Classical series expansion for Sylvester-type equations (Theorem \rom{7}.2.2 in \cite{bhatia2013matrix}) also leads to equation \eqref{eq:cor_low_rank}.
\end{remark}
}
\subsection{Examples of Using Theorem \ref{thm:main2}}\label{sec:3.4}
Theorem \ref{thm:main2} is used to derive the refined $\ell_{2,\infty}$ bound (Theorem \ref{thm:2toinf_full}) and the sin$\Theta$ bound between singular vectors and their resided singular subspace (Theorem \ref{thm:single}), which provided the main intuition behind our PCA and singular value truncation results (Theorem \ref{thm:pca} and Theorem \ref{thm:thresh}) in Section \ref{sec:2}. Here, we only present the proof of Theorem \ref{thm:single}, and defer the proof of the rest to the appendix and the supplementary material since they are more involved.
\begin{proof}
Again we denote $Y\coloneqq  U_2^T \widetilde{U}_1$ and $Z\coloneqq V_2^T \widetilde{V}_1$.  Restricting \eqref{eq:implicit} in Theorem \ref{thm:main2} to the $j$th columns ($1\leq j\leq r$), we have
\[
Y_j = (C_1)_j+\widetilde\sigma_j(F_U^{21})_j\circ(\alpha_{22}Z_j)+(F_U^{21})_j\circ(\Sigma_2\alpha_{22}^T Y_j),
\]
\[
Z_j=(C_2)_j+\widetilde\sigma_j (F_V^{21})_j\circ(\alpha_{22}^T Y_j)+(F_V^{21})_j\circ(\Sigma_2^T\alpha_{22}Z_j).
\]
It is easy to verify that $\|(C_1)_j\|\leq\frac{\widetilde\sigma_j}{\widetilde\sigma_j^2-\sigma_{r+1}^2}\|\alpha_{21}\|+\frac{\sigma_{r+1}}{\widetilde\sigma_j^2-\sigma_{r+1}^2}\|\alpha_{12}\|,\ \|(C_2)_j\|\leq\frac{\sigma_{r+1}}{\widetilde\sigma_j^2-\sigma_{r+1}^2}\|\alpha_{21}\|+\frac{\widetilde\sigma_j}{\widetilde\sigma_j^2-\sigma_{r+1}^2}\|\alpha_{12}\|$, then 
\[
\|Y_j\|\leq \frac{1}{\widetilde\sigma_j^2-\sigma_{r+1}^2} \left( \widetilde\sigma_j\|\alpha_{21}\|+\sigma_{r+1}\|\alpha_{12}\|+\widetilde\sigma_j\|\alpha_{22}\|\|Z_j\|+\sigma_{r+1}\|\alpha_{22}\|\|Y_j\|\right),
\]
\[
\|Z_j\|\leq\frac{1}{\widetilde\sigma_j^2-\sigma_{r+1}^2} \left(\sigma_{r+1}\|\alpha_{21}\|+\widetilde\sigma_j\|\alpha_{12}\|+\widetilde\sigma_j\|\alpha_{22}\|\|Y_j\|+\sigma_{r+1}\|\alpha_{22}\|\|Z_j\|\right).
\]
Summing up the first inequality multiplied by  $\widetilde{\sigma}_j^2-\sigma_{r+1}^2-\sigma_{r+1}\|\alpha_{22}\|$ and the second inequality multiplied by $\widetilde\sigma_j\|\alpha_{22}\|$, after some simplification we get
\begin{align*}
    \|Y_j\|=\|U_2^T\widetilde  u_j\|&\leq\frac{\widetilde\sigma_j\|\alpha_{21}\|+\sigma_{r+1}\|\alpha_{12}\|+\|\alpha_{22}\|\|\alpha_{12}\|}{\widetilde\sigma_j^2-(\sigma_{r+1}+\|\alpha_{22}\|)^2}\\
    &\leq \frac{(\sigma_j-\|\Delta A\|)\|\Delta A\|+\sigma_{r+1}\|\Delta A\|+\|\Delta A\|^2}{(\sigma_j-\|\Delta A\|)^2-(\sigma_{r+1}+\|\Delta A\|)^2}\\
    &\leq\frac{\|\Delta A\|}{\sigma_j-\sigma_{r+1}-2\|\Delta A\|}\\
    &\leq \frac{3\|\Delta A\|}{\sigma_j-\sigma_{r+1}},
\end{align*}
provided that $3\|\Delta A\| \leq  \sigma_r -\sigma_{r+1}$. Here the second inequality is because the upper bound on the right hand side is decreasing with respect to $\widetilde\sigma_j$ and increasing with respect to $\|\alpha_{22}\|$. Similarly, we also have

\[
\|Z_j\|=\|V_2^T\widetilde v_j\|\leq\frac{\widetilde\sigma_j\|\alpha_{12}\|+\sigma_{r+1}\|\alpha_{21}\|+\|\alpha_{22}\|\|\alpha_{21}\|}{\widetilde\sigma_j^2-(\sigma_{r+
1}+\|\alpha_{22}\|)^2} \leq \frac{3\|\Delta A\|}{\sigma_j-\sigma_{r+1}}.
\]
\end{proof}

\section{Proof of the main results}
{In Section \ref{subsec:thm31}, we derive the proof of Theorem \ref{thm:main} and Lemma \ref{lemma:Hnorm}. After that, we present the proof of Theorem \ref{thm:2toinf_full} in Section \ref{sec:two_to_infinity}. 
Since the proof of one key lemma (Lemma \ref{lemma:full_rank}) is long and involved, we divide it into low-rank case and full-rank case. We prove the low-rank case in Section \ref{subsec:lemma44}, and the proof of full-rank case is deferred to appendix. In Section \ref{subsec:thm29} we provide the proof of Theorem \ref{thm:thresh}, while the proof of Theorem \ref{thm:pca} can be found in  Section \ref{subsec:thm27}.}
\subsection{Proof of Theorem \ref{thm:main} and Lemma \ref{lemma:Hnorm}}\label{subsec:thm31}
\begin{proof}[Proof of Theorem \ref{thm:main}]
First, decompose the perturbation $\Delta A$ in the following two ways :
\begin{equation}\label{eq:dA1}
\begin{aligned}
    \Delta A&=\widetilde A-A=\widetilde U\widetilde\Sigma\widetilde V^T-U\Sigma V^T  \\
    &=(U+\Delta U)\widetilde \Sigma\widetilde V^T-U\Sigma(\widetilde V-\Delta V)^T  \\
    &=U\widetilde \Sigma\widetilde V^T+(\Delta U)\widetilde \Sigma\widetilde V^T-U\Sigma \widetilde V^T+U\Sigma(\Delta V)^T \\
    &=U(\Delta \Sigma)\widetilde V^T+(\Delta U)\widetilde \Sigma\widetilde V^T+U\Sigma(\Delta V)^T,   
\end{aligned}
\end{equation}
and
\begin{equation}\label{eq:dA2}
\begin{aligned}
    \Delta A&=\widetilde A-A=\widetilde U\widetilde\Sigma\widetilde V^T-U\Sigma V^T\\
    &=\widetilde U\widetilde \Sigma(V+\Delta V)^T-(\widetilde U-\Delta U)\Sigma V^T\\
    &=\widetilde U\widetilde \Sigma V^T+\widetilde U\widetilde \Sigma(\Delta V)^T-\widetilde U\Sigma V^T+(\Delta U)\Sigma V^T\\
    &=\widetilde U(\Delta \Sigma) V^T+\widetilde U\widetilde \Sigma(\Delta V)^T+(\Delta U)\Sigma V^T. 
\end{aligned}
\end{equation}
Multiplying \eqref{eq:dA1} with $U^T$ on the left and $\widetilde V$ on the right leads to
\begin{equation}\label{eq:dP1}
U^T(\Delta A)\widetilde V=\Delta \Sigma+U^T(\Delta U)\widetilde\Sigma+\Sigma(\Delta V)^T\widetilde V.
\end{equation}
Similarly, multiplying \eqref{eq:dA2} with $\widetilde U^T$ on the left and $V$ on the right we obtain
\begin{equation}\label{eq:dP2}
    \widetilde U^T(\Delta A)V=\Delta\Sigma+\widetilde\Sigma(\Delta V)^T V+\widetilde U^T(\Delta U)\Sigma.
\end{equation}
Denote $dP=U^T(\Delta A)\widetilde V,\ d\bar P=\widetilde U^T(\Delta A)V,\ \Delta \Omega_U=U^T(\Delta U),\ \Delta\Omega_V=V^T(\Delta V)$. Notice that $I = \widetilde U^T\widetilde U=U^T U$ gives $(U+\Delta U)^T\widetilde U=U^T(\widetilde U-\Delta U)$, hence $U^T\Delta U=-\Delta U^T\widetilde U$. Similarly, we also have $V^T\Delta V=-\Delta V^T\widetilde V$. Plugging these into \eqref{eq:dP1} and \eqref{eq:dP2}, we have
\begin{equation}\label{eq:aligned}
\left\{
\begin{aligned}
dP&=U^T\Delta A\widetilde V = \Delta\Sigma+\Delta\Omega_U\widetilde\Sigma-\Sigma\Delta\Omega_V,\\
d\bar P&=\widetilde U^T \Delta A V = \Delta \Sigma+\widetilde \Sigma\Delta\Omega_V^T-\Delta\Omega_U^T\Sigma.
\end{aligned}
\right.
\end{equation}
Next, from \eqref{eq:aligned} we can cancel $\Delta\Omega_V$ by
\begin{align*}
    G_U:&=dP\widetilde\Sigma^T+\Sigma d\bar P^T\\&=\Delta\Sigma\widetilde \Sigma^T+\Sigma(\Delta\Sigma)^T+\Delta\Omega_U\widetilde\Sigma\widetilde\Sigma^T-\Sigma\Sigma^T\Delta\Omega_U\\&=\widetilde\Sigma\widetilde\Sigma^T-\Sigma\Sigma^T+\Delta\Omega_U\widetilde\Sigma\widetilde\Sigma^T-\Sigma\Sigma^T\Delta\Omega_U.
\end{align*}
Let $\Delta\Omega_U=\{w_{ij}\}_{i,j=1}^n$, then for all $1\leq i,j\leq n$, the following equations hold
\begin{equation}\label{eq:GU}
(G_U)_{ij}=\left\{
\begin{aligned}
&(\widetilde\sigma_j^2-\sigma_i^2)w_{ij},& i\neq j,\\
&(\widetilde\sigma_j^2-\sigma_i^2)(w_{ij}+1),&i=j.
\end{aligned}
\right.
\end{equation}
Here if $i>\min\{n,m\}$, we define $\sigma_i$ or $\widetilde\sigma_i$ to be 0. Also, define $F_U^{12},\ F_U^{21},\ F_V^{12},\ F_V^{21}$ as in the statement of Theorem \ref{thm:main}. By assumption, $\widetilde\sigma_r-\sigma_{r+1}>0,\ \sigma_r-\widetilde\sigma_{r+1}>0$, we can directly check that the denominators in these four matrices only have nonzero entries, thus are well defined. Consider the upper right part in $\Delta \Omega_U=U^T(\Delta U)$, that is, $1\leq i\leq r,\ r+1\leq j\leq n$, from \eqref{eq:GU} we have 
\[
w_{ij}=\frac{1}{\widetilde\sigma_j^2-\sigma_i^2}(G_U)_{ij},\ 1\leq i\leq r, \ r+1\leq j\leq n.
\]
Therefore,
\[
U_1^T\widetilde U_2=U_1^T(\Delta U_2)=F_U^{12}\circ (G_U^{12})=F_U^{12}\circ (U_1^T(\Delta A)\widetilde V_2\widetilde\Sigma_2^T+\Sigma_1 V_1^T(\Delta A)^T\widetilde U_2).
\]
Following the same reasoning, we also obtain
\[
U_2^T\widetilde U_1=U_2^T(\Delta U_1)=F_U^{21}\circ(U_2^T(\Delta A)\widetilde V_1\widetilde\Sigma_1^T+\Sigma_2 V_2^T(\Delta A)^T\widetilde U_1),
\]
\[
V_1^T\widetilde V_2=V_1^T(\Delta V_2)=F_V^{12}\circ(\Sigma_1^T U_1^T(\Delta A)\widetilde V_2+ V_1^T(\Delta A)^T\widetilde U_2\widetilde\Sigma_2),
\]
\[
V_2^T\widetilde V_1=V_2^T(\Delta V_2)=F_V^{21}\circ(\Sigma_2^T U_2^T(\Delta A)\widetilde V_1+ V_2^T(\Delta A)^T\widetilde U_1\widetilde\Sigma_1).
\]
\end{proof}

\begin{proof}[Proof of Lemma \ref{lemma:Hnorm}]
Here we only prove the first inequality in \eqref{eq:Hnorm_ineq}, i.e., $|||F_U^{21}\circ (H_1\widetilde\Sigma_1) ||| \leq\frac{\widetilde\sigma_r}{\widetilde\sigma_r^2-\sigma_{r+1}^2}|||H_1|||$, the other three inequalities can be proved similarly. Recall the definition of $F_U^{21}$ is $(F_U^{21})_{i-r,j}=\frac{1}{\widetilde\sigma_j^2-\sigma_i^2},\ r+1\leq i\leq n,\ 1\leq j\leq r$. We directly have 
\[
F_U^{21}\circ(H_1\widetilde\Sigma_1)=\bar F_U^{21}\circ H_1, 
\]
where
\[
(\bar F_U^{21})_{i-r,j}=\frac{\widetilde\sigma_j}{\widetilde\sigma_j^2-\sigma^2_i},\ r+1\leq i\leq n,\ 1\leq j\leq r.
\]
Let  $B_1= F_U^{21}\circ(H_1\widetilde\Sigma_1)$, then $H_1=\widetilde F_U^{21}\circ B_1$, where
\[
(\widetilde F_U^{21})_{i-r,j}=\frac{\widetilde\sigma_j^2-\sigma^2_i}{\widetilde\sigma_j}=\widetilde\sigma_j-\frac{\sigma_i^2}{\widetilde\sigma_j},\ r+1\leq i\leq n,\ 1\leq j\leq r.
\]
Inserting the above expression of $\widetilde F $ into $H_1=\widetilde F_U^{21}\circ B_1$, we have
\begin{align*}
 H_1=   B_1\begin{pmatrix}
\widetilde \sigma_1 &\ &\ &\ \\
\ &\widetilde\sigma_2 &\ &\ \\
\ &\ & \ddots \\
\ &\ &\ &\ \widetilde\sigma_r
\end{pmatrix}-\begin{pmatrix}
\sigma_{r+1}^2 &\ &\ &\ \\
\ &\sigma_{r+2}^2 &\ &\ \\
\ &\ & \ddots \\
\ &\ &\ &\ \sigma_n^2
\end{pmatrix}B_1\begin{pmatrix}
\frac{1}{\widetilde\sigma_1} &\ &\ &\ \\
\ &\frac{1}{\widetilde\sigma_2} &\ &\ \\
\ &\ & \ddots \\
\ &\ &\ &\ \frac{1}{\widetilde\sigma_r}
\end{pmatrix}.
\end{align*}
Take norm on both sides, we obtain 
\[
||| H_1|||\geq  \widetilde\sigma_r||| B_1|||-\frac{\sigma_{r+1}^2}{\widetilde \sigma_r}||| B_1|||=\frac{\widetilde\sigma_r^2-\sigma_{r+1}^2}{\widetilde\sigma_r}||| B_1|||,
\]
which further gives $|||B_1|||\leq\frac{\widetilde\sigma_r}{\widetilde\sigma_r^2-\sigma_{r+1}^2}||| H_1|||$. 
\end{proof}
\begin{remark}\label{rmk:main}
When $\widetilde\sigma_r>\sigma_{r+1},\ \sigma_r>\widetilde\sigma_{r+1}$,the bounds in Lemma \ref{lemma:Hnorm} are tight. That is, in this case, there exists $H_i,\ 1\le i\leq 4$, such that the equalities in \eqref{eq:Hnorm_ineq} and \eqref{eq:Hnorm_ineq2} hold. Specifically, let 
\[
H_1=\begin{pmatrix}
0 & \cdots & 0 & \epsilon\\
0 &\cdots & \cdots & 0\\
\vdots & \vdots & \vdots & \vdots\\
0 & \cdots &\cdots & 0 
\end{pmatrix}\in\mathbb{R}^{(n-r)\times r},\ H_3=\begin{pmatrix}
0 & \cdots & \cdots & 0\\
\vdots &\vdots & \vdots & \vdots\\
0 & \vdots & \vdots & \vdots\\
 \epsilon & 0 &\cdots & 0 
\end{pmatrix}\in\mathbb{R}^{r\times (m-r)},
\]
\[
H_2=\begin{pmatrix}
0 & \cdots & 0 & \epsilon\\
0 &\cdots & \cdots & 0\\
\vdots & \vdots & \vdots & \vdots\\
0 & \cdots &\cdots & 0 
\end{pmatrix}\in\mathbb{R}^{(m-r)\times r},\ H_4=\begin{pmatrix}
0 & \cdots & \cdots & 0\\
\vdots &\vdots & \vdots & \vdots\\
0 & \vdots & \vdots & \vdots\\
 \epsilon & 0 &\cdots & 0 
\end{pmatrix}\in\mathbb{R}^{r\times (n-r)},
\]
then we can directly check that the equalities in \eqref{eq:Hnorm_ineq} and \eqref{eq:Hnorm_ineq2} hold.
\end{remark}

\subsection{Proof of Theorem \ref{thm:2toinf_full}}\label{sec:two_to_infinity}

To prove Theorem \ref{thm:2toinf_full}, we need to decompose $\widetilde{U}_1 -U_1Q$ into a sum of several components and bound them separately. For convenience, we put the decomposition in the following lemma, which is similar in nature to Theorem 3.1 in \cite{cape2019two}.
\begin{proposition}\label{thm:expandU}Set the rotation $Q$ to be $Q = Q_1Q_2^T$, where $Q_1$ and $Q_2$ are the left and right singular vectors from the SVD: $U_1^T\widetilde{U}_1 = Q_1SQ_2^T$, then
\begin{equation}\label{eq:decomposition}
\widetilde{U}_1 -U_1Q =U_2U_2^T\Delta A V_1V_1^T\widetilde V_1\widetilde\Sigma_1^{-1}+U_2U_2^T\Delta A V_2V_2^T\widetilde V_1\widetilde\Sigma_1^{-1}+U_2 \Sigma_2 V_2^T \widetilde V_1\widetilde\Sigma_1^{-1}+U_1Q_1(S-I)Q_2^T,
\end{equation}
and
\begin{equation}\label{eq:S}
\|S-I\|\leq \|\sin\Theta(U_1,\widetilde{U}_1)\|^2.
\end{equation}
\end{proposition}
\begin{proof}
By direct calculation,  we have
\begin{align}\label{eq:dec}
    \widetilde U_1-U_1 Q&=\widetilde U_1-U_1 Q_1Q_2^T\notag\\
    &=\widetilde U_1-U_1Q_1SQ_2^T+U_1Q_1(S-I)Q_2^T\notag\\
    &=\widetilde U_1-U_1U_1^T\widetilde U_1+U_1Q_1(S-I)Q_2^T\notag\\
    &=U_2U_2^T\widetilde U_1+U_1Q_1(S-I)Q_2^T\notag\\
    &=U_2U_2^T\Delta A \widetilde V_1\widetilde\Sigma_1^{-1}+U_2 \Sigma_2 V_2^T \widetilde V_1\widetilde\Sigma_1^{-1} +U_1Q_1(S-I)Q_2^T\\
  &=U_2U_2^T\Delta A V_1V_1^T\widetilde V_1\widetilde\Sigma_1^{-1}+U_2U_2^T\Delta A V_2V_2^T\widetilde V_1\widetilde\Sigma_1^{-1}+U_2 \Sigma_2 V_2^T \widetilde V_1\widetilde\Sigma_1^{-1}+U_1Q_1(S-I)Q_2^T\notag. 
\end{align}

In addition, since $\|S\|=\|U_1^T\widetilde U_1\|\leq 1$, $$\|S-I\|=1-\min_i S_i\leq {1-\min_i S_i^2}=\|\sin\Theta(U_1,\widetilde U_1)\|^2,$$
where $S_i$ is the $i$th diagonal entry of $S$. Hence
\[
\|U_1 Q_1(S-I)Q_2^T\|_{2,\infty}\leq\|U_1\|_{2,\infty}\|S-I\| \leq \|U_1\|_{2,\infty}\|\sin\Theta(U_1,\widetilde U_1)\|^2.
\]
\end{proof}
The first and the last terms in the expansion \eqref{eq:decomposition} are easy to bound, the following lemma is devoted to bounding the middle terms, which requires invoking the angular perturbation formula Theorem \ref{thm:main2}. 
\begin{lemma}\label{lemma:full_rank}
Under the assumption of Theorem \ref{thm:2toinf_full}, it holds that
\begin{equation*}
    \max\{\|U_2U_2^T\Delta AV_2 V_2^T\widetilde V_1\widetilde\Sigma_1^{-1}\|_{2,\infty},\|U_2\Sigma_2 V_2^T\widetilde V_1\widetilde\Sigma_1^{-1}\|_{2,\infty}\}\leq C\frac{\sigma R(r, n)}{\sigma_r(A)-\sigma_{r+1}(A)},
\end{equation*}
where $C$ is some constant and 
$$
R(r, n) = \left\{
\begin{aligned}
&\sqrt r+\sqrt{\log  n},\quad & \textrm{if $A$ is of rank $r$}; \\
& r+\sqrt{r \log n}, & \textrm{else}.
\end{aligned}\right.
$$
\end{lemma}
Before proving this lemma, let us first see how to use it to prove Theorem \ref{thm:2toinf_full}.

\begin{proof}[\textbf{Proof of Theorem \ref{thm:2toinf_full}}]Due to \eqref{eq:decomposition}, we have
\begin{align*}
\min_{\widetilde{Q}\in \mathbb{O}_r} \|\widetilde{U}_1 -U_1 \widetilde{Q}\|_{2,\infty} &\leq \underbrace{\|U_2U_2^T\Delta A V_1V_1^T\widetilde{V}_1\widetilde\Sigma_1^{-1}\|_{2,\infty}}_{(\RN{1})}+\|U_2U_2^T\Delta A V_2 V_2^T\widetilde V_1\widetilde\Sigma_1^{-1}\|_{2,\infty}\notag\\
&\quad\ +\|U_2\Sigma_2 V_2^T\widetilde V_1\widetilde\Sigma_1^{-1}\|_{2,\infty}+\underbrace{\|U_1\|_{2,\infty}\|\sin\Theta(U_1,\widetilde{U}_1)\|^2}_{(\RN{2})}.
\end{align*}
The two middle terms are bounded in Lemma \ref{lemma:full_rank} . We are left to bound the first and the last terms. For the last term, we have
\begin{align*}
(\RN{2}) 
& \leq \|U_1\|_{2,\infty}\left(\frac{2\|\Delta A\|}{\sigma_r(A)-\sigma_{r+1}(A)}\right)^2\leq \|U_1\|_{2,\infty}\frac{36\sigma^2 \bar n}{(\sigma_r(A)-\sigma_{r+1}(A))^2}.
\end{align*}
Here the first inequality used \eqref{eq:uniform} in Lemma \ref{lm:ufoneside}, the second one used Corollary 7.3.3 of \cite{vershynin2018high} which bounds the spectral norm of i.i.d. Gaussian matrices: with probability at least $1-e^{-c\bar n}$ for some absolute constant $c$, $\|\Delta A\|\leq 3\sigma\sqrt{\bar n}$.

Next we bound $(\RN{1})$.
\begin{align}\label{eq:Ipart1}
\begin{split}
(\RN{1}) &= \|U_2U_2^T\Delta A V_1V_1^T\widetilde{V}_1\widetilde\Sigma_1^{-1}\|_{2,\infty} \leq \|U_2U_2^T\Delta A V_1\|_{2,\infty}\|V_1^T\widetilde{V}_1\widetilde\Sigma_1^{-1}\|\\
&\leq \frac{1}{\sigma_r(\widetilde{A})}\|U_2U_2^T\Delta A V_1\|_{2,\infty}\leq \frac{7}{6\sigma_r(A)}\|U_2U_2^T\Delta A V_1\|_{2,\infty}.
\end{split}
\end{align}
Here the last inequality is by Weyl's bound and the assumption $\sigma_r(A)>21\sigma\sqrt{\bar n}$. \eqref{eq:Ipart1} implies that it suffices to bound the row norms of $U_2U_2^T\Delta A V_1$. Since $\Delta A$ is i.i.d. $N(0,\sigma^2)$, $U_2$ and $V_1$ are independent of $\Delta A$ and that $\|U_2\|=\|V_1\|= 1$, then each row of $U_2U_2^T\Delta A V_1$ is a Gaussian vector having independent Gaussian entries with mean 0 and variance at most $\sigma^2$. By exactly the same proof as Theorem 3.1.1 in \cite{vershynin2018high}, there exists a constant $c$ such that for all $t>0$,
\[
\mathbb{P}(\left|\|u_i^TU_2^T\Delta A V_1\|-\sigma\|u_i^TU_2^T\|\sqrt{r}\right|\geq t)<2e^{-\frac{ct^2}{\sigma^2\|u_i^TU_2^T\|^2}},
\]
where $u_i^T$ is the $i$th row of $U_2$.




Setting in the above $t=\sigma\|u_i^TU_2^T\|\sqrt{3\log n/c}$, then with probability at least $1-\frac{2}{n^3}$,
\[
\|u_i^TU_2^T\Delta A V_1\|\leq c_1\sigma(\sqrt{r}+\sqrt{\log n})\|u_i^TU_2^T\|\leq c_1\sigma(\sqrt{r}+\sqrt{\log n}),
\]
with some constant $c_1$. By the union bound, the probability of failure for all the rows is at most $\frac{2}{n^2}$. Hence with probability at least $1-\frac{2}{n^2}$, it holds 
\[
\|U_2U_2^T\Delta A V_1\|_{2,\infty} \leq c_1\sigma(\sqrt{r}+\sqrt{\log n}).
\]
Plugging this into \eqref{eq:Ipart1}, we obtain
\[
(\RN{1})\leq \frac{c_1\sigma(\sqrt r+\sqrt{\log n})}{\sigma_r(A)}.
\]
Combining the bounds on $\RN{1}$, $\RN{2}$ and Lemma \ref{lemma:full_rank} completes the proof.
\end{proof}
\subsection{Proof of  Lemma \ref{lemma:full_rank}}\label{subsec:lemma44}
Here we first provide the proof for the low-rank case to give the reader some intuition. The full-rank case follows a similar idea but is quite notationally heavy, {we defer the proof of Lemma \ref{lemma:full_rank} for full-rank case to appendix.}

\begin{proof}[\textbf{Proof of Lemma \ref{lemma:full_rank}- the low-rank case}]
When $A$ is of rank $r$, the second quantity to be bounded in Lemma \ref{lemma:full_rank} is 0, hence we focus on the first quantity $\|U_2U_2^T\Delta A V_2 V_2^T\widetilde V_1\widetilde\Sigma_1^{-1}\|_{2,\infty}$.

Let $u_i^T$ be the $i$th row of $U_2$, then by Corollary \ref{lemma:Z}, the $i$th row of $U_2U_2^T\Delta A V_2V_2^T\widetilde V_1\widetilde\Sigma_1^{-1}$ can be expressed as
\begin{align}
u_i^TU_2^T\Delta A V_2V_2^T\widetilde V_1\widetilde\Sigma_1^{-1}&=u_i^T\alpha_{22}\left(\sum_{k=0}^{\infty}(\alpha_{22}^T\alpha_{22})^k(\alpha_{12}^TU_1^T\widetilde U_1\widetilde\Sigma_1^{-1}+\alpha_{22}^T\alpha_{21}V_1^T\widetilde V_1\widetilde\Sigma_1^{-2})(\widetilde\Sigma_1^{-2})^k\right)\widetilde\Sigma_1^{-1}\notag\\
&=u_i^T\left(\sum_{k=0}^{\infty}(\alpha_{22}\alpha_{22}^T)^k(\alpha_{22}\alpha_{12}^TU_1^T\widetilde U_1\widetilde\Sigma_1^{-1}+\alpha_{22}\alpha_{22}^T\alpha_{21}V_1^T\widetilde V_1\widetilde\Sigma_1^{-2})(\widetilde\Sigma_1^{-2})^k\right)\widetilde\Sigma_1^{-1}, \label{eq:II}
\end{align}
where $\alpha_{ij} = U_i^T\Delta AV_j$. Due to the orthogonality of $U$ and $V$, the entries in each $\alpha_{ij}$ follow i.i.d. $N(0,\sigma^2)$ distribution, and $\alpha_{22}$ is independent of $\alpha_{12}$. This further implies that $u_i^T(\alpha_{22}\alpha_{22}^T)^k\alpha_{22}$ and $u_i^T(\alpha_{22}\alpha_{22}^T)^k\alpha_{22}\alpha_{22}^T$ are independent of $\alpha_{12}$ and $\alpha_{21}$, respectively. Conditional on $\alpha_{22}$,  $u_i^T(\alpha_{22}\alpha_{22}^T)^k\alpha_{22}\alpha_{12}^T$ varies with $\alpha_{12}$, and it follows normal distribution. Again by Theorem 3.1.1 in \cite{vershynin2018high}, for fixed $k=0,...$, there exists a constant $c$ such that 
\[
\mathbb{P}(\left|\|u_i^T(\alpha_{22}\alpha_{22}^T)^k\alpha_{22}\alpha_{12}^T\|-\sigma\sqrt{r}\|u_i^T(\alpha_{22}\alpha_{22}^T)^k\alpha_{22}\|\right|>t)\leq 2\exp\left(-\frac{ct^2}{\sigma^2\|u_i^T(\alpha_{22}\alpha_{22}^T)^k\alpha_{22}\|^2}\right).
\]
Setting in the above 
$t=\sigma\|u_i^T(\alpha_{22}\alpha_{22}^T)^k\alpha_{22}\|\sqrt{\log(n^3\cdot 2^k)/c}
$, we get with probability at least $1-\frac{2}{2^kn^3}$,
\begin{align}\label{eq:singleterm}
\|u_i^T(\alpha_{22}\alpha_{22}^T)^k\alpha_{22}\alpha_{12}^T\|&\leq \sigma\sqrt{r}\|u_i^T(\alpha_{22}\alpha_{22}^T)^k\alpha_{22}\|+t \notag \\& \leq c_2 \sigma\|u_i^T(\alpha_{22}\alpha_{22}^T)^k\alpha_{22}\|(\sqrt{r}+\sqrt{\log n}+\sqrt{k}),
\end{align}
where $c_2$ is some absolute constant.
Then 
\begin{align*}
\|u_i^T(\alpha_{22}\alpha_{22}^T)^k\alpha_{22}\alpha_{12}^TU_1^T\widetilde U_1\widetilde\Sigma_1^{-(2k+2)}\|&\leq c_2\frac{\sigma}{\widetilde\sigma_r}\left(\frac{\|\alpha_{22}\|}{\widetilde\sigma_r}\right)^{2k+1}(\sqrt{r}+\sqrt{\log n}+\sqrt{k}).
\end{align*}
Let $\lambda =\frac{\|\alpha_{22}\|}{\widetilde\sigma_r}$. We next argue that $\lambda<1/2$.  By Corollary 7.3.3 of \cite{vershynin2018high},  $\|\Delta A\|\leq  3\sigma\sqrt{\bar n}$ with probability at least $1-e^{-c \bar n}$. On this event, by Weyl's bound,
$$\widetilde{\sigma}_r \geq \sigma_r-\|\Delta A\| \geq\sigma_r- 3\sigma\sqrt{\bar n} \geq 18\sigma\sqrt {\bar n} \geq 6\|\Delta A\| \geq 6\|\alpha_{22}\|,$$ 
which implies $\lambda<1/2$.  The third inequality above is due to the assumption $21\sigma\sqrt {\bar n}<\sigma_r$.  By union bound on the probability of failure of \eqref{eq:singleterm} over all $k=0,...$, we have with probability at least $1-\frac{4}{n^3}$,
\begin{align*}
\sum_{k=0}^{\infty}\|u_i^T(\alpha_{22}\alpha_{22}^T)^k\alpha_{22}\alpha_{12}^TU_1^T\widetilde U_1\widetilde\Sigma_1^{-(2k+2)}\|_2&\leq c_3\frac{\sigma}{\widetilde\sigma_r}\left( \sum_{k=0}^{\infty} \sqrt k \lambda^{2k+1} + (\sqrt{\log n}+\sqrt{r})\sum_{k=0}^{\infty} \lambda^{2k+1}\right)\\& \leq c_4\frac{\sigma}{\widetilde\sigma_r}(\sqrt r+\sqrt{\log n}) \\
& \leq c_5\sigma \frac{\sqrt r+\sqrt{\log n}}{\sigma_r(A)}, 
\end{align*}
with $c_3-c_5$ being absolute constants, where the last inequality used Weyl's bound and the assumption $\sigma_r(A)>21\sigma\sqrt{\bar n}$, 
and the second inequality used the fact that for any $0<\lambda<1/2$, we have
\begin{equation}\label{eq:exp_sum}
\sum_{k=0}^{\infty}  \sqrt k \lambda^{2k+1}\leq \sum_{k=0}^{\infty} \sqrt k \lambda^k \leq \sum_{k=1}^{\infty} (k+1) \lambda^k = \frac{d}{d\lambda}\left(\frac{1}{1-\lambda}\right)-1 \leq \frac{2\lambda}{(1-\lambda)^2} < 4.
\end{equation}
Following the same reasoning, with probability at least $1-\frac{4}{n^3}$,
\[
\sum_{k=0}^{\infty}\|u_i^T(\alpha_{22}\alpha_{22}^T)^{(k+1)}\alpha_{21}V_1^T\widetilde V_1\widetilde\Sigma_1^{-(2k+3)}\|_2\leq c_6\sigma \frac{\sqrt r+\sqrt{\log n}}{\sigma_r(A)},
\]
for some constant $c_6$. Using these in \eqref{eq:II}, by the union bound, we obtain that with probability at least $1-\frac{8}{n^2}$,
\[
\|U_2U_2^T\Delta A V_2V_2^T\widetilde V_1\widetilde\Sigma_1^{-1}\|_{2,\infty} \leq c_7\sigma \frac{\sqrt r+\sqrt{\log n}}{\sigma_r(A)},
\]
where $c_7$ is some constant.
\end{proof}

\subsection{Proof of Theorem \ref{thm:thresh}}\label{subsec:thm29}
Although Theorem \ref{thm:thresh} is motivated and could be proved by Theorem \ref{thm:single}, we provide an alternative proof that is more straightforward. For this purpose, 
we will need the following lemmas.
\begin{lemma}\label{lm:threh} Under the same assumption as Theorem \ref{thm:thresh}, we have
\[
A_r-\widetilde{A}_r=
U\begin{pmatrix} -U_1^T\Delta A\widetilde{V}_1 & -U_1^T\Delta A\widetilde{V}_2 \\ - U_2^TAV_2V_2^T\widetilde{V}_1&  0 \end{pmatrix}\widetilde{V}^T+U\begin{pmatrix} 0 & U_1^T\widetilde{U}_2\widetilde{U}_2^T\widetilde{A}\widetilde{V}_2 \\ -U_2^T\Delta A\widetilde{V}_1 &  0 \end{pmatrix}\widetilde{V}^T.
\]
\end{lemma}
\begin{proof}
The lemma can be straightforwardly verified by using the relation $A_r=U_1\Sigma_1V_1^T$ and  $\widetilde{A}_r=\widetilde{U}_1\widetilde{\Sigma}_1\widetilde{V}_1^T$. 
\end{proof}
\begin{lemma}[Lemma 2 in \cite{luo2021schatten}]\label{lemma:copy}
Suppose $x_1\geq x_2\geq...\geq x_k\geq 0$ and $y_1\geq y_2\geq...\geq y_k\geq 0$. For any $1\leq j\leq k,\ \sum_{i=1}^j x_i\leq\sum_{i=1}^j y_i$. Then for any $p\geq 1$,
\[
\sum_{i=1}^k x_i^p\leq\sum_{i=1}^k y_i^p.
\]
The equality holds if and only if $(x_1,x_2,...,x_k)=(y_1,y_2,...,y_k)$.
\end{lemma}
\begin{lemma}[Theorem 1 in \cite{thompson1975singular}]\label{ap:B}Assume $A,B,C=A+B$ are (not necessarily square) matrices of the same size, with singular values
\[
\alpha_1\geq\alpha_2\geq...,\ \beta_1\geq\beta_2\geq...,\ \gamma_1\geq\gamma_2\geq...,
\]
respectively. Let $i_1<i_2<...<i_m$ and $j_1<j_2<...<j_m$ be positive integers, and set
\[
k_t=i_t+j_t-t,\ t=1,2,...,m.
\]
Then the singular values of $A,B,C$ satisfy 
\[
\sum_{t=1}^m\gamma_{k_t}\leq\sum_{t=1}^m
\alpha_{i_t}+\sum_{t=1}^m\beta_{j_t}.
\]
\end{lemma}

\begin{lemma}\label{lm:ufoneside}
We have the following uniform error bound on $\sin\Theta$ distance
\begin{equation}
\max\{\|\sin\Theta(U_1,\widetilde U_1)\|,\|\sin\Theta(V_1,\widetilde V_1)\|\} \leq \min\left\{\frac{2\|\Delta A\|}{\sigma_r-\sigma_{r+1}},1\right\}.
\label{eq:uniform}
\end{equation}
\end{lemma}
\begin{proof}[Proof of Lemma \ref{lm:ufoneside}]

If $\sigma_r=\sigma_{r+1}$, \eqref{eq:uniform} holds trivially, here we consider the case $\sigma_r>\sigma_{r+1}$. Consider the two possibilities $\sigma_r-\sigma_{r+1}>2\|\Delta A\|$ and $\sigma_r-\sigma_{r+1}\leq 2\|\Delta A\|$. When $\sigma_r-\sigma_{r+1}>2\|\Delta A\|$, this and the Weyl's bound 
\[
|\widetilde\sigma_r-\sigma_r|\leq\|\Delta A\|,
\]
together give
\[\widetilde\sigma_r-\sigma_{r+1}>\sigma_r-\sigma_{r+1}-\|\Delta A\|>\frac{1}{2}(\sigma_r-\sigma_{r+1})>0,
\]
which ensures the assumption in Theorem \ref{thm:two-sided} to hold, and then \eqref{eq:uniform_hat} in Theorem \ref{thm:two-sided} implies 
\[
\|\sin\Theta(U_1,\widetilde U_1)\|=\|U_2^T\widetilde U_1\|\leq\frac{\|\Delta A\|}{\sigma_r-\|\Delta A\|-\sigma_{r+1}}\leq\frac{2\|\Delta A\|}{\sigma_r-\sigma_{r+1}}.
\]
When $\sigma_r-\sigma_{r+1}\leq 2\|\Delta A\|$, we directly have
$$\|U_2^T\widetilde U_1\|\leq 1\leq \frac{2\|\Delta A\|}{\sigma_r-\sigma_{r+1}}.$$ Putting the two cases together, we have \[\|\sin\Theta(U_1,\widetilde U_1)\|\leq\min\{\frac{2\|\Delta A\|}{\sigma_r-\sigma_{r+1}},1\}.
\]
Following the same reasoning, we also have $\|\sin\Theta(V_1,\widetilde V_1)\|\leq\min\{\frac{2\|\Delta A\|}{\sigma_r-\sigma_{r+1}},1\}$, thus \eqref{eq:uniform} holds.
\end{proof}  

\begin{proof}[\textbf{Proof of Theorem \ref{thm:thresh}}]
By Lemma \ref{lm:threh}, 
\begin{align}
\|A_r-\widetilde A_r\| & \leq \left\|\begin{pmatrix} -U_1^T\Delta A\widetilde{V}_1 & -U_1^T\Delta A\widetilde{V}_2 \\ - U_2^TAV_2V_2^T\widetilde{V}_1&  0 \end{pmatrix}\right\| +\left\|\begin{pmatrix} 0 & U_1^T\widetilde{U}_2\widetilde{U}_2^T\widetilde{A}\widetilde{V}_2 \\ -U_2^T\Delta A\widetilde{V}_1 &  0 \end{pmatrix}\right\| \notag \\
& \leq \sqrt{\|U_1^T\Delta A\|^2+\|U_2^TAV_2V_2^T\widetilde{V}_1\|^2}+ \max\{\|U_2^T\Delta A\widetilde{V}_1\|,\|U_1^T\widetilde{U}_2\widetilde{U}_2^T\widetilde{A}\widetilde{V}_2 \| \}  
\notag \\
&= \sqrt{\|U_1^T \Delta A\|^2+\tau^2}+ \max\{\|U_2^T\Delta A\widetilde{V}_1\|,\nu  \}  \label{eq:A_r},
\end{align}
where we have let $\nu =\|U_1^T\widetilde{U}_2\widetilde{U}_2^T\widetilde{A}\widetilde{V}_2 \|$ and $\tau=\|U_2^TAV_2V_2^T\widetilde{V}_1\|$, we next bound $\tau$ and $\nu$.
\begin{align}\label{eq:alpha}
\nu = \|U_1^T\widetilde{U}_2\widetilde{U}_2^T\widetilde{A}\widetilde{V}_2 \| =\|U_1^T\widetilde{U}_2\widetilde{\Sigma}_2\| \leq \widetilde{\sigma}_{r+1}\|U_1^T\widetilde{U}_2\|. 
\end{align}
Due to the Weyl's bound, we also have \begin{equation}\label{eq:weyl}
    |\widetilde{\sigma}_{r+1}-\sigma_{r+1}|<\|\Delta A\|.
\end{equation}
By \eqref{eq:uniform},
\[
\nu\leq(\sigma_{r+1}+\|\Delta A\|)\min\left\{\frac{2\|\Delta A\|}{\sigma_r-\sigma_{r+1}},1\right\}\leq\|\Delta A\|+\sigma_{r+1}\min\left\{\frac{2\|\Delta A\|}{\sigma_r-\sigma_{r+1}},1\right\}.
\]
Similarly, we can also derive
\[
\tau=\|\Sigma_2 V_2^T\widetilde V_1\|\leq\sigma_{r+1}\min\left\{\frac{2\|\Delta A\|}{\sigma_r-\sigma_{r+1}},1\right\}.
\]
Inserting the upper bounds of $\tau$ and $\nu$ back to \eqref{eq:A_r} completes the proof of \eqref{eq:2norm_threholding}.
For Frobenius norm:
\begin{align*}
\|A_r-\widetilde{A}_r\|_F^2 &=\left\|\begin{pmatrix} -U_1^T\Delta A\widetilde{V}_1 & -U_1^T\Delta A\widetilde{V}_2 +U_1^T\widetilde{U}_2\widetilde{U}_2^T\widetilde{A}\widetilde{V}_2 \\ -U_2^T\Delta A\widetilde{V}_1 - U_2^TAV_2V_2^T\widetilde{V}_1&  0 \end{pmatrix}\right\|_F^2 & \\ 
&= \underbrace{\left\|\begin{pmatrix} -U_1^T\Delta A\widetilde{V}_1  \\ -U_2^T\Delta A\widetilde{V}_1 - U_2^TAV_2V_2^T\widetilde{V}_1 \end{pmatrix}\right\|_F^2}_{\coloneqq R_1}+ \underbrace{\left\|\begin{pmatrix} -U_1^T\Delta A\widetilde{V}_2 +U_1^T\widetilde{U}_2\widetilde{U}_2^T\widetilde{A}\widetilde{V}_2 \\  0 \end{pmatrix}\right\|_F^2}_{\coloneqq R_2} .
\end{align*}
Next we bound $R_1$ and $R_2$ separately. First consider $R_2$, let $M\Lambda W^T$ be the singular value decomposition of $U_1^T\widetilde U_2$, where $M\in\mathbb{R}^{r\times r},\ \Lambda\in\mathbb{R}^{r\times r},\ W\in\mathbb{R}^{(n-r)\times r}$. Then
\begin{align*}
R_2&=\|-U_1 \Delta A \widetilde V_2+U_1\widetilde U_2\widetilde\Sigma_2\|_F^2\\
&\leq 2\|U_1^T\Delta A\widetilde{V}_2\|_F^2+2\|U_1^T\widetilde{U}_2\widetilde\Sigma_2\|_F^2\\
&\leq 2\|(\Delta A)_r\|_F^2+2\|U_1^T\widetilde U_2 WW^T\widetilde\Sigma_2\|_F^2\\
&\leq 2\|(\Delta A)_r\|_F^2+2\left(\min\left\{\frac{2\|\Delta A\|}{\sigma_r-\sigma_{r+1}},1\right\}\right)^2\|W^T\widetilde\Sigma_2\|_F^2\\
&\leq 2\|(\Delta A)_r\|_F^2+2\left(\min\left\{\frac{2\|\Delta A\|}{\sigma_r-\sigma_{r+1}},1\right\}\right)^2\sum_{k=1}^r\sigma^2_{r+k}(\widetilde A).
\end{align*}
In the second to last inequality, we used the fact that $\|AB\|_F\leq \|A\|\|B\|_F$ and in the last inequality, we used $\|P_{\Omega}A\|_F\leq \|A_r\|_F$ for any $r$-dimensional subspace $\Omega$.
By Lemma \ref{ap:B},
\[
\sum_{i=1}^k \sigma_{r+i}(\widetilde A)=\sum_{i=1}^k \sigma_{r+i}(A+\Delta A)\leq \sum_{i=1}^k\sigma_i(\Delta A)+\sum_{i=1}^k\sigma_{r+i}(A)=\sum_{i=1}^k\left(\sigma_i(\Delta A)+\sigma_{r+i}(A)\right), 1\leq k\leq r.
\]
From  Lemma \ref{lemma:copy}, we have 
\[
\sum_{k=1}^r\sigma_{r+k}^2(\widetilde A)\leq \sum_{k=1}^r\left(\sigma_k(\Delta A)+\sigma_{r+k}(A)\right)^2\leq \left(\|(\Delta A)_r\|_F+\|(\Sigma_2)_r\|_F\right)^2.
\]
Hence 
\begin{equation}\label{eq:C2}
R_2\leq 2\|(\Delta A)_r\|_F^2+2\left(\min\left\{\frac{2\|\Delta A\|}{\sigma_r-\sigma_{r+1}},1\right\}\right)^2\left(\|(\Delta A)_r\|_F+\|(\Sigma_2)_r\|_F\right)^2.
\end{equation}
Next we consider $R_1$. Notice that 
\[
R_1=\left\|{\begin{pmatrix} -U_1^T\Delta A\widetilde{V}_1  \\ -U_2^T\Delta A\widetilde{V}_1  \end{pmatrix}}+{\begin{pmatrix}0\\-U_2^T A V_2V_2^T\widetilde V_1\end{pmatrix}}\right\|_F^2\leq(\|(\Delta A)_r\|_F+\|\Sigma_2 V_2^T\widetilde V_1\|_F)^2.
\]
Following the same reasoning as in bounding $R_2$, we have 
\begin{equation}\label{eq:C1}
R_1\leq \left(\|(\Delta A)_r\|_F+\|(\Sigma_2)_r\|_F\min\left\{\frac{2\|\Delta A\|}{\sigma_r-\sigma_{r+1}},1\right\}\right)^2.
\end{equation}
Combining \eqref{eq:C2} and \eqref{eq:C1}, we obtain
\[
\|A_r-\widetilde A_r\|_F^2\leq 2\|(\Delta A)_r\|_F^2+3\left(\|(\Delta A)_r\|_F+\|(\Sigma_2)_r\|_F\min\left\{\frac{2\|\Delta A\|}{\sigma_r-\sigma_{r+1}},1\right\}\right)^2.
\]
\end{proof}


\subsection{Proof of Theorem \ref{thm:pca}}\label{subsec:thm27}
\begin{proof}[Proof of Theorem \ref{thm:pca}] Let $\widetilde{V}_1^TV_1=Q_1SQ_2^T$ be the SVD of $\widetilde{V}_1^TV_1$. Define a special rotation $\hat{Q}=Q_1Q_2^T$, and we bound $|||U_1\Sigma_1-\widetilde{U}_1\widetilde{\Sigma}_1\hat{Q}|||$, where $|||\cdot|||$ can be either the the spectral or the Frobenius norm. This yields an upper bound on $\min_{Q\in \mathbb{O}_r}|||U_1\Sigma_1-\widetilde{U}_1\widetilde{\Sigma}_1{Q}|||$. By a direct calculation,
\begin{align}
|||U_1\Sigma_1-\widetilde{U}_1\widetilde{\Sigma}_1\hat{Q}|||&\leq |||U_1\Sigma_1-\widetilde{U}_1\widetilde{\Sigma}_1\widetilde{V}_1^TV_1+\widetilde{U}_1\widetilde{\Sigma}_1\widetilde{V}_1^TV_1-\widetilde{U}_1\widetilde{\Sigma}_1\hat{Q}||| \notag\\ &\leq
|||U_1\Sigma_1-\widetilde{U}_1\widetilde{\Sigma}_1\widetilde{V}_1^TV_1|||+|||\widetilde{U}_1\widetilde{\Sigma}_1\widetilde{V}_1^TV_1-\widetilde{U}_1\widetilde{\Sigma}_1\hat{Q}||| \notag\\
& = |||(U_1\Sigma_1V_1^T-\widetilde{U}_1\widetilde{\Sigma}_1\widetilde{V}_1^T)V_1|||+|||\widetilde{U}_1\widetilde{\Sigma}_1(\widetilde{V}_1^TV_1-\hat{Q})||| \notag\\
& \leq |||U_1\Sigma_1V_1^T-\widetilde{U}_1\widetilde{\Sigma}_1\widetilde{V}_1^T|||+|||\widetilde{U}_1\widetilde{\Sigma}_1(\widetilde{V}_1^TV_1-\hat{Q})|||\notag\\
& = |||A_r-\widetilde{A}_r|||+|||\widetilde{U}_1\widetilde{\Sigma}_1(\widetilde{V}_1^TV_1-\hat{Q})|||. \label{eq:pca_inter}
\end{align}
The first term of \eqref{eq:pca_inter} can be bounded by Theorem \ref{thm:thresh}. Let us focus on the second term.
Let $Z =V_2^T\widetilde{V}_1$. Observe $Z^T Z+\widetilde{V}_1^TV_1(\widetilde{V}_1^TV_1)^T=I_r$, this implies $\widetilde{V}_1^TV_1 = \sqrt{I_r-Z^T Z}\hat{Q}$ (Lemma \ref{lm:sqrt}), where the square root of a positive semi-definite matrix $B$ is defined to be the positive semi-definite  matrix $\widetilde{B}$ such that $\widetilde{B}\widetilde{B}=B$. 

Using this observation on the quantity inside the norm of the second term on the right hand side of \eqref{eq:pca_inter}, we have
\[
\widetilde{U}_1\widetilde{\Sigma}_1(\widetilde{V}_1^TV_1-\hat{Q})=\widetilde{U}_1\widetilde{\Sigma}_1(\sqrt{I-Z^T Z}-I)\hat{Q}=-\widetilde{U}_1\widetilde{\Sigma}_1Z^T Z(\sqrt{I-Z^T Z}+I)^{-1}\hat{Q},
\]
where the last equality used the fact that $(\sqrt{I-Z^T Z}-I)(\sqrt{I-Z^T Z}+I)=-Z^T Z$.
Hence 
\begin{equation}\label{eq:PCA_second}
|||\widetilde{U}_1\widetilde{\Sigma}_1(\widetilde{V}_1^TV_1-\hat{Q})|||\leq |||\widetilde{U}_1\widetilde{\Sigma}_1Z^T Z|||\cdot\|(\sqrt{I-Z^T Z}+I)^{-1}\hat{Q}\| \leq |||\widetilde{\Sigma}_1Z^T|||.
\end{equation}
The last inequality used $\|Z\|\leq 1$, and $\|(\sqrt{I-Z^T Z}+I)^{-1}\|\leq 1$.
Notice that
\begin{align*}
    (\widetilde\Sigma_1 Z^T)^T=Z\widetilde\Sigma_1=V_2^T\widetilde V_1\widetilde\Sigma_1=V_2^T{\widetilde A}^T\widetilde U_1=V_2^T  \Delta A^T\widetilde U_1+V_2^T A^T\widetilde U_1=V_2^T \Delta A^T\widetilde U_1+\Sigma_2^T U_2^T\widetilde U_1.
\end{align*}
Then for the spectral norm of $\widetilde\Sigma_1 Z^T$, we have
\begin{align*}
    \|\widetilde\Sigma_1 Z^T\|\leq\|\Delta A\|+\sigma_{r+1}\|U_2^T\widetilde U_1\|
    \leq \|\Delta A\|+\sigma_{r+1}\min\left\{\frac{2\|\Delta A\|}{\sigma_r-\sigma_{r+1}},1\right\}.
\end{align*}
For the Frobenius norm, we have
\[
\|\widetilde\Sigma_1 Z^T\|_F\leq\|V_2^T \Delta A^T\widetilde U_1\|_F+\|\Sigma_2 U_2^T\widetilde U_1\|_F\leq\|(\Delta A)_r\|_F+\|(\Sigma_2)_r\|_F\min\left\{\frac{2\|\Delta A\|}{\sigma_r-\sigma_{r+1}},1\right\}.
\]
Combining this with \eqref{eq:pca_inter} and \eqref{eq:PCA_second} completes the proof.
\end{proof}
\begin{lemma}\label{lm:sqrt}
Let $A\in \mathbb{R}^{r\times r}$ be a semi-definite matrix with eigenvalues no greater than 1, $B\in \mathbb{R}^{r\times r}$ has SVD $B=U_BS_B V_B^T$. In addition, $A$ and $B$ satisfy $A+BB^T=I$, then
\[
B = \sqrt{I-A}U_B V_B^T.
\]
\end{lemma}
\begin{proof}
Since $BB^T = U_BS_B^2U_B^T$, then $\sqrt{BB^T}= U_BS_BU_B^T$, and therefore $B = \sqrt{BB^T}U_B V_B^T$. By assumption, $BB^T=I-A$, then the result of the lemma follows.
\end{proof}

\normalem
\bibliographystyle{plain}
\bibliography{revision}
\section{Appendix}

\begin{proof}[\textbf{Proof of Lemma \ref{lemma:full_rank}- the full-rank case}]We first bound $\|U_2U_2^T\Delta A V_2V_2^T\widetilde V_1\widetilde\Sigma_1^{-1}\|_{2,\infty}$. To do so, we need Theorem \ref{thm:main2} to obtain the expansion of $V_2^T\widetilde V_1$. Let us first check that in the setting of this lemma (i.e., $21\sigma\sqrt{\bar n}<\sigma_r(A)-\sigma_{r+1}(A)$), the condition in  Theorem \ref{thm:main2} is satisfied with high probability, that is, $\|\mathcal{F}\|<1$, where
\[
\mathcal{F}\left(\begin{matrix}C_1\\C_2\end{matrix}\right) = \left(\begin{matrix}F_U^{21}\circ(\Sigma_2\alpha_{22}^TC_1)+F_U^{21}\circ(\alpha_{22}C_2\widetilde\Sigma_1^T)\\F_V^{21}\circ(\alpha_{22}^TC_1\widetilde\Sigma_1)+F_V^{21}\circ(\Sigma_2^T\alpha_{22}C_2)\end{matrix}\right).
\]
As discussed in the proof of Theorem \ref{thm:2toinf_full}, $\|\Delta A\|\leq 3\sigma\sqrt{\bar n}$ with probability at least $1-e^{-c \bar n}$ with some constant $c$. By the assumption $21\sigma\sqrt{\bar n}<\sigma_r(A)-\sigma_{r+1}(A)$, we have with probability at least $1-e^{-c\bar n}$,
\begin{align*}
&\quad\left\|\mathcal{F}\left(\begin{matrix}C_1\\C_2\end{matrix}\right)\right\|\\
&\leq\|F_U^{21}\circ(\Sigma_2\alpha_{22}^TC_1)\|+\|F_U^{21}\circ(\alpha_{22}C_2\widetilde\Sigma_1^T)\|+\|F_V^{21}\circ(\alpha_{22}^TC_1\widetilde\Sigma_1)\|+\|F_V^{21}\circ(\Sigma_2^T\alpha_{22}C_2)\|\\
&\leq \frac{\sigma_{r+1}}{\widetilde\sigma_r^2-\sigma_{r+1}^2}\|\alpha_{22}\|\|C_1\|+\frac{\widetilde\sigma_{r}}{\widetilde\sigma_r^2-\sigma_{r+1}^2}\|\alpha_{22}\|\|C_2\|+\frac{\widetilde\sigma_{r}}{\widetilde\sigma_r^2-\sigma_{r+1}^2}\|\alpha_{22}\|\|C_1\|+\frac{\sigma_{r+1}}{\widetilde\sigma_r^2-\sigma_{r+1}^2}\|\alpha_{22}\|\|C_2\|\\
&=\frac{\|\alpha_{22}\|}{\widetilde\sigma_r-\sigma_{r+1}}(\|C_1\|+\|C_2\|)\\
&\leq\frac{\|\Delta A\|}{\sigma_r-\sigma_{r+1}-\|\Delta A\|}\cdot{\sqrt{2(\|C_1\|^2+\|C_2\|^2)}}\\
&\leq\frac{\sqrt{2}}{6}\sqrt{\|C_1\|^2+\|C_2\|^2}.
\end{align*}
Here the second inequality is due to Lemma \ref{lemma:Hnorm}. Now we have $\|\mathcal{F}\|<1$ with high probability, which enables us to use Theorem \ref{thm:main2} to decompose $U_2U_2^T\Delta AV_2V_2^T\widetilde V_1\widetilde\Sigma_1^{-1}$. Denote 
\[
a_1=F_U^{21}\circ(\Sigma_2\alpha_{12}^TU_1^T\widetilde U_1), \ a_2=F_U^{21}\circ(\alpha_{21}V_1^T\widetilde V_1\widetilde\Sigma_1^T),\ a_3=F_V^{21}\circ(\alpha_{12}^T U_1^T\widetilde U_1\widetilde\Sigma_1),\ a_4=F_V^{21}\circ(\Sigma_2^T\alpha_{21} V_1^T\widetilde V_1),
\]
\[
f_1(X)=F_U^{21}\circ(\Sigma_2\alpha_{22}^T X),\ f_2(X)=F_U^{21}\circ(\alpha_{22} X \widetilde \Sigma_1^T),\ f_3(X)=F_V^{21}\circ(\alpha_{22}^T X\widetilde\Sigma_1),\ f_4(X)=F_V^{21}\circ(\Sigma_2^T\alpha_{22}X).
\]
By Theorem \ref{thm:main2}, each term in the expansion of $V_2^T\widetilde V_1$ is of the form 
\[
f_{i_1}(f_{i_2}(...(f_{i_k}(a_{i_0})))),\ 1\leq i_0,i_1,...,i_k\leq 4,\ k=0,1,2,....
\]
{\bf Now assume $i_1,...,i_k$ and $k$ are fixed}. 
Let $w$ be the $i$th row in $U_2\alpha_{22}f_{i_1}(f_{i_2}(...(f_{i_k}(a_{i_0}))))\widetilde\Sigma_1^{-1}$, and let $b^T=u_i^T\alpha_{22}$, where $u_i^T$ is the $i$th row of $U_2$. Then $w=b^Tf_{i_1}(f_{i_2}(...(f_{i_k}(a_{i_0}))))\widetilde\Sigma_1^{-1}$. Notice that $a_{i_0}$ and each $f_{i_s},\ 1\leq s\leq k$, either contains $\Sigma_2$ or $\widetilde\Sigma_1$, let $h_{i_s}=1$ if   $f_{i_s}$ contains $\Sigma_2$, and $h_{i_s}=0$ if $f_{i_s}$ contains $\widetilde\Sigma_1$. Let $h_{i_0}=1$  if $a_{i_0}$ contains $\Sigma_2$ and $h_{i_s}=0$ if $a_{i_0}$ contains $\widetilde\Sigma_1$. Also, let $m$ be the total number of times that $\widetilde\Sigma_1$ appears in $f_{i_s}$ and $a_{i_0}$. Then \begin{equation}
\label{eq:c}h_{i_0}+h_{i_1}+...+h_{i_k}+m=k+1.
\end{equation} 
Likewise, each $f_{i_s}, \ 1\leq s\leq k$ either contains $\alpha_{22}$ or $\alpha_{22}^T$. Let $d_{i_s}=\alpha_{22}$ if $f_{i_s}$ contains $\alpha_{22}$ and $d_{i_s}=\alpha_{22}^T$ if it contains $\alpha_{22}^T$. Also, let $\gamma=\alpha_{12}^T$ if $a_{i_0}$ contains $\alpha_{12}^T$ and $\gamma=\alpha_{21}$ if $a_{i_0}$ contains $\alpha_{21}$. Last, denote $\beta=V_1^T\widetilde V_1$ if  $a_{i_0}$ contains $V_1^T\widetilde V_1$  and $\beta=U_1^T\widetilde U_1$, if $a_{i_0}$ contains $U_1^T\widetilde U_1$. For the $\gamma$ and $\beta$ defined above, let $\gamma_l^T$ be the $l$th row of $\gamma$, i.e., $\gamma=[\gamma_1^T;\gamma_2^T;...;\gamma_{n-r}^T]$ and $\beta_i$ be the $i$th column of $\beta$, i.e., $\beta=[\beta_1,\beta_2,...,\beta_r]$. Then for $1\leq j\leq r$, the $j$th entry in $w$ is
\begin{align*}
w_j&=\sum_{l_1}\frac{b_{l_1}\sigma_{r+l_1}^{h_{i_1}}}{\widetilde\sigma_j^2-\sigma_{{r+l_1}}^2}\sum_{l_2}\frac{(d_{i_1})_{l_1l_2}\sigma_{r+l_2}^{h_{i_2}}}{\widetilde\sigma_j^2-\sigma_{r+l_2}^2}...\sum_{l_{k}}\frac{(d_{i_{k-1
}})_{l_{k-1}l_{k}}\sigma_{r+l_{k}}^{h_{i_{k}}}}{\widetilde\sigma_j^2-\sigma_{r+l_{k}}^2}\sum_{l_{0}}\frac{(d_{i_{k
}})_{l_{k}l_{0}}\sigma_{r+l_{0}}^{h_{i_{0}}}}{\widetilde\sigma_j^2-\sigma_{r+l_{0}}^2}\gamma_{l_0}^T(\widetilde\sigma_j^{m-1}\beta_j)\\
&=\langle\underbrace{\sum_{l_1}\frac{b_{l_1}\sigma_{r+l_1}^{h_{i_1}}}{\widetilde\sigma_j^2-\sigma_{{r+l_1}}^2}\sum_{l_2}\frac{(d_{i_1})_{l_1l_2}\sigma_{r+l_2}^{h_{i_2}}}{\widetilde\sigma_j^2-\sigma_{r+l_2}^2}...\sum_{l_{k}}\frac{(d_{i_{k-1
}})_{l_{k-1}l_{k}}\sigma_{r+l_{k}}^{h_{i_{k}}}}{\widetilde\sigma_j^2-\sigma_{r+l_{k}}^2}\sum_{l_{0}}\frac{(d_{i_{k
}})_{l_{k}l_{0}}\sigma_{r+l_{0}}^{h_{i_{0}}}}{\widetilde\sigma_j^2-\sigma_{r+l_{0}}^2}\gamma_{l_0}}_{\equiv M_j},\widetilde\sigma_j^{m-1}\beta_j\rangle.
\end{align*}
In the above, let \[
M_j = \sum_{l_1}\frac{b_{l_1}\sigma_{r+l_1}^{h_{i_1}}}{\widetilde\sigma_j^2-\sigma_{{r+l_1}}^2}\sum_{l_2}\frac{(d_{i_1})_{l_1l_2}\sigma_{r+l_2}^{h_{i_2}}}{\widetilde\sigma_j^2-\sigma_{r+l_2}^2}...\sum_{l_{k}}\frac{(d_{i_{k-1
}})_{l_{k-1}l_{k}}\sigma_{r+l_{k}}^{h_{i_{k}}}}{\widetilde\sigma_j^2-\sigma_{r+l_{k}}^2}\sum_{l_{0}}\frac{(d_{i_{k
}})_{l_{k}l_{0}}\sigma_{r+l_{0}}^{h_{i_{0}}}}{\widetilde\sigma_j^2-\sigma_{r+l_{0}}^2}\gamma_{l_0}.
\]
We first bound $\|M_j\|$. Denote $\eta_j=\sigma_j^2-\widetilde\sigma_j^2$, and $\Delta_{jl}=\sigma_j^2-\sigma_{r+l}^2$. Notice that 
\begin{align*}
\frac{1}{\widetilde\sigma_j^2-\sigma_{r+l}^2}&=\frac{1}{\sigma_j^2-\sigma_{r+l}^2+(\widetilde\sigma_j^2-\sigma_j^2)}=\frac{1}{(\sigma_j^2-\sigma_{r+l}^2)(1 +\frac{\widetilde\sigma_j^2-\sigma_j^2}{\sigma_j^2-\sigma_{r+l}^2})} \\ &=\frac{1}{\Delta_{jl}(1-\frac{\eta_j}{\Delta_{jl}})}= \frac{1}{\Delta_{jl}}(1+\frac{\eta_j}{\Delta_{jl}}+(\frac{\eta_j}{\Delta_{jl}})^2+...).
\end{align*}
Hence 
\begin{align*}
    \|M_j\|&=\left\|\sum_{l_1,..., l_k,l_0}\frac{b_{l_1}\prod_{s=1}^k(d_{i_s})_{l_{s}l_{s+1}}\prod_{s=0}^k\sigma_{r+l_s}^{h_{i_s}}}{\prod_{s=0}^k{\Delta_{jl_s}}}\prod_{s=0}^k\left(1+\frac{\eta_j}{\Delta_{jl_s}}+(\frac{\eta_j}{\Delta_{jl_s}})^2+...\right){\gamma_{l_0}}\right\|\\
    &=\left\|\sum_{q_0,q_1,...,q_k=0}^\infty\sum_{l_1,..., l_k,l_0}\frac{b_{l_1}\prod_{s=1}^k(d_{i_s})_{l_{s}l_{s+1}}\prod_{s=0}^k\sigma_{r+l_s}^{h_{i_s}}}{\big(\prod_{s=0}^k{\Delta_{jl}\big)\prod_{s=0}^k\Delta_{jl_s}^{q_s}}}{\gamma_{l_0}}\cdot\eta_j^{\sum_{s=0}^k q_s}\right\|\\
    &\leq\sum_{q_0,q_1,...,q_k=0}^\infty\underbrace{\left\|\sum_{l_1,..., l_k,l_0}\frac{b_{l_1}\prod_{s=1}^k(d_{i_s})_{l_{s}l_{s+1}}\prod_{s=0}^k\sigma_{r+l_s}^{h_{i_s}}}{\prod_{s=0}^k{\Delta_{jl_s}^{1+q_s}}}{\gamma_{l_0}}\right\|}_{T(j,q_0,...,q_k)}\cdot \left|\eta_j^{\sum_{s=0}^k q_s}\right|.
\end{align*}
Here we let $l_{k+1}=l_0$.  In the above, denote 
\[
T(j,q_0,...,q_k)\equiv\sum_{l_1,..., l_k,l_0}\frac{b_{l_1}\prod_{s=1}^k(d_{i_s})_{l_{s}l_{s+1}}\prod_{s=0}^k\sigma_{r+l_s}^{h_{i_s}}}{\prod_{s=0}^k{\Delta_{j l_s}^{1+q_s}}}{\gamma_{l_0}}\in\mathbb{R}^{r\times 1}.
\] Next, we bound the $\ell_2$ norm of $T(j,q_0,...,q_k)$. 
Notice that we can rewrite $T(j,q_0,...,q_k)$ in the following way
\[
T(j,q_0,...,q_k)=\left(b^T\hat f_{i_1}^{q_1} \hat f_{i_2}^{q_2}... \hat f_{i_k}^{q_{k}}\hat a_{i_0}^{q_0}\right)^T,\ 1\leq i_0,i_2,...,i_k\leq 4,\ k\geq 0.
\]
Here matrices $\hat f_{i_s}$ are modified from of the functions $f_{i_s}$, and matrix $\hat a_{i_0}$ is modified from the function $a_{i_0}$. Explicitly,
\[
\hat a_1^{q}=\hat F_U^{21,q}\Sigma_2\alpha_{12}^T, \ \hat a_2^{q}=\hat F_U^{21,q}\alpha_{21},\ \hat a_3^{q}=\hat F_V^{12,q}\alpha_{12}^T ,\ \hat a_4^{q}=\hat F_V^{21,q}\Sigma_2^T\alpha_{21},
\]
\[
\hat f_1^q=\hat F_U^{21,q}\Sigma_2\alpha_{22}^T,\ \hat f_2^q=\hat F_U^{21,q}\alpha_{22},\ \hat f_3^q=\hat F_V^{21,q}\alpha_{22}^T ,\ \hat f_4^q=
\hat F_V^{21,q}\Sigma_2^T\alpha_{22},
\]
where $\hat F_U^{21,q}$ and $\hat F_V^{21,q}$ are diagonal matrices with diagonal entries
\[
(\hat F_U^{21,q})_{i'-r,i'-r}=\frac{1}{(\sigma_{j}^2-\sigma_{i'}^2)^{1+q}},\ r+1\leq i'\leq n,
\]
\[
(\hat F_V^{21,q})_{i'-r,i'-r}=\frac{1}{(\sigma_{j}^2-\sigma_{i'}^2)^{1+q}},\ r+1\leq i'\leq m.
\]
Similar as in Theorem \ref{thm:main}, if $i'>\min\{n,m\}$, we define $\sigma_{i'}$ to be 0.

As before, In the above expression of $T(j,q_0,...,q_k)$, $
\hat a^{q_0}_{i_0}$ either contains $\alpha_{21}$ or $\alpha_{12}^T$. If it contains the former, let $\gamma$ be the former, and it contains the latter, let $\gamma$ be the latter. It is easy to check that this $\gamma$ coincides with the $\gamma$ defined in the paragraph under \eqref{eq:c}.

Conditional on $\alpha_{22}$, the only random variable in $T(j,q_0,...,q_k)$ is $\alpha_{21}$ or $\alpha_{12}^T$, that is $\gamma$. Therefore, if we write $T(j,q_0,...,q_k)=G(\gamma)^T$, then the linear operator $G$ is independent of $\gamma$, and it is straightforward to check that
\[
\|G\|\leq\underbrace{\|b\|\frac{\|\alpha_{22}\|^k\sigma_{r+1}^{\sum_{s=0}^k h_{i_s}}}{(\sigma_j^2-\sigma_{r+1}^2)^{k+1+\sum_{s=0}^k q_s}}}_{\equiv K},\ 1\leq j\leq r.
\]
Again, conditional on $\alpha_{22}$, since for each $1\leq p\leq r$, the $p$th entry of $T(j,q_0,...,q_k)$ only depends on the $p$th column of $\alpha_{21}$ or $\alpha_{12}^T$,  then different entries of $T(j,q_0,...,q_k)$ are independent of each other,  
 each following a Gaussian distribution $\mathcal{N}(0,\sigma^2\xi_p^2)$, and $\xi_p\leq K,\ 1\leq p\leq r $, where $K$ denotes the above bound.  By Theorem 3.1.1 in \cite{vershynin2018high}, there exists some constant $c$ such that
\[
\mathbb{P}(\|T(j,q_0,...,q_k)\|-K\sigma\sqrt{r}\geq t)\leq 2e^{-ct^2/\sigma^2K^2}.
\]
Let $t =K\sigma\sqrt{\log (n^3\cdot 2^{\sum_{s=0}^k q_s}\cdot 8^k r)/c}$, then with probability at least $1-\frac{2}{n^3\cdot 2^{\sum_{s=0}^k q_s}\cdot 8^kr}$,
\begin{align*}
\|T(j,q_0,...,q_k)\|\leq c_1\sigma\|b\|\frac{\sigma_{r+1}^{\sum_{s=0}^k h_{i_s}}\|\alpha_{22}\|^k}{(\sigma_j^2-\sigma_{r+1}^2)^{k+1+\sum_{s=0}^k q_s}}(\sqrt{\log n}+\sum_{s=0}^k \sqrt{q_s}+\sqrt{k}+\sqrt{r}),
\end{align*}
where $c_1$ is some constant. Hence 
\begin{align*}
    \|M_j\|&\leq \sum_{q_0,q_1,...,q_k=0}^\infty\left\|T(j,q_0,...,q_k)\right\|\cdot\left|\eta_j^{\sum_{s=0}^k q_s}\right|\\
    &\leq c_1\sigma \frac{\|b\|\sigma_{r+1}^{\sum_{s=0}^k h_{i_s}}\|\alpha_{22}\|^k}{(\sigma_j^2-\sigma_{r+1}^2)^{k+1}}\sum_{q_0,q_1,...,q_k=0}^\infty(\sqrt{\log n}+\sum_{s=0}^k\sqrt{ q_s}+\sqrt{k}+\sqrt{r})\cdot \left|\frac{\eta_j}{\sigma_j^2-\sigma_{r+1}^2}\right|^{\sum_{s=0}^k q_s}\\
    &\leq c_1\sigma \frac{\|b\|\sigma_{r+1}^{\sum_{s=0}^k h_{i_s}}\|\alpha_{22}\|^k}{(\sigma_j^2-\sigma_{r+1}^2)^{k+1}}\left((\sqrt{\log n}+\sqrt{k}+\sqrt{r})\frac{1}{(1-\frac{|\eta_j|}{\sigma_j^2-\sigma_{r+1}^2})^{k+1}}+\frac{2(k+1)}{(1-\frac{|\eta_j|}{\sigma_j^2-\sigma_{r+1}^2})^{k+1}}\right)\\
    &=c_2\sigma \frac{\|b\|\sigma_{r+1}^{\sum_{s=0}^k h_{i_s}}\|\alpha_{22}\|^k}{(\sigma_j^2-\sigma_{r+1}^2-|\eta_j|)^{k+1}}(\sqrt{\log n}+\sqrt{r}+k),
\end{align*}
where $c_2$ is a constant.
In the second inequality above, we used the fact that
\begin{align*}
\sum_{q_0,q_1,...,q_k=0}^\infty\left|\frac{\eta_j}{\sigma_j^2-\sigma_{r+1}^2}\right|^{\sum_{s=0}^k q_s} &=\sum_{q_0,q_1,...,q_{k-1}=0}^\infty\left|\frac{\eta_j}{\sigma_j^2-\sigma_{r+1}^2}\right|^{\sum_{s=0}^{k-1} q_s}\sum_{q_k=0}^\infty\left|\frac{\eta_j}{\sigma_j^2-\sigma_{r+1}^2}\right|^{q_k}\\
&=\frac{1}{(1-\frac{|\eta_j|}{\sigma_j^2-\sigma_{r+1}^2})}\sum_{q_0,q_1,...,q_{k-1}=0}^\infty\big|\frac{\eta_j}{\sigma_j^2-\sigma_{r+1}^2}\big|^{\sum_{s=0}^{k-1} q_s}\\
&=...=\frac{1}{(1-\frac{|\eta_j|}{\sigma_j^2-\sigma_{r+1}^2})^{k+1}}.
\end{align*}
and that
\begin{align*}
&\sum_{q_0,q_1,...,q_k=0}^\infty(\sum_{s=0}^k\sqrt{q_s})\big|\frac{\eta_j}{\sigma_j^2-\sigma_{r+1}^2}\big|^{\sum_{s=0}^k q_s}\\ &=\sum_{s=0}^k\sum_{q_0,...,q_{s-1},q_{s+1},...,q_k=0}^\infty\big|\frac{\eta_j}{\sigma_j^2-\sigma_{r+1}^2}\big|^{\sum_{l\neq s}q_l}\sum_{q_s=0}^\infty\sqrt{q_s}\big|\frac{\eta_j}{\sigma_j^2-\sigma_{r+1}^2}\big|^{q_s}\\
&\leq\sum_{s=0}^k\frac{2}{(1-\frac{|\eta_j|}{\sigma_j^2-\sigma_{r+1}^2})}\sum_{q_0,...,q_{s-1},q_{s+1},...,q_k=0}^\infty\big|\frac{\eta_j}{\sigma_j^2-\sigma_{r+1}^2}\big|^{\sum_{l\neq s} q_l}\\
&=\frac{2(k+1)}{(1-\frac{|\eta_j|}{\sigma_j^2-\sigma_{r+1}^2})^{k+1}}.
\end{align*}
Here the inequality is due to \eqref{eq:exp_sum}, which holds under the condition 
\begin{align*}
    \frac{|\eta_j|}{\sigma_j^2-\sigma_{r+1}^2}&=\frac{|\widetilde\sigma_j^2-\sigma_j^2|}{\sigma_j^2-\sigma_{r+1}^2}\leq \frac{(\sigma_j+\|\Delta A\|)^2-\sigma_j^2}{\sigma_j^2-\sigma_{r+1}^2}\\
    &\leq\frac{2\sigma_j\|\Delta A\|+\|\Delta A\|^2}{(\sigma_j-\sigma_{r+1})(\sigma_j+\sigma_{r+1})}
    \leq \frac{(2\sigma_j+\frac{1}{7}\sigma_j)\|\Delta A\|}{7\|\Delta A\|\sigma_j}<\frac{1}{2}.
\end{align*}
Therefore with probability at least $1-\frac{4}{n^3\cdot 4^k r}$,
\begin{align*}
\|w_j\|&\leq\|M_j\|\cdot\widetilde\sigma_j^{m-1}\|\beta_j\|\\
&\leq c_2\sigma \frac{\|b\|\widetilde\sigma_j^{m-1}\sigma_{r+1}^{\sum_{s=0}^k h_{i_s}}\|\alpha_{22}\|^k}{(\sigma_j^2-\sigma_{r+1}^2-|\eta_j|)(\sigma_j^2-\sigma_{r+1}^2-|\eta_j|)^k}(\sqrt{\log n}+\sqrt{r}+k)\\
&\leq c_2\sigma\frac{\|b\|(\frac{8}{7}\sigma_j)^k\|\Delta A\|^k}{\frac{2}{3}(\sigma_j^2-\sigma_{r+1}^2)(\frac{34}{7}\sigma_j\|\Delta A\|)^k}(\sqrt{\log n}+\sqrt{r}+k)\\
&\leq \frac{3}{2}c_2\sigma\frac{\|b\|}{\sigma_r^2-\sigma_{r+1}^2}(\frac{4}{17})^k(\sqrt{\log n}+\sqrt{r}+k).
\end{align*}
Here the third inequality is due to $\sum_{s=0}^k h_{i_s}+m-1=k$ and 
\begin{align*}
\sigma_j^2-\sigma_{r+1}^2-|\widetilde\sigma_j^2-\sigma_j^2|&\geq\sigma_j(\sigma_j-\sigma_{r+1})-((\sigma_j+\|\Delta A\|)^2-\sigma_j^2)\\
&\geq7\sigma_j\|\Delta A\|-2\sigma_j\|\Delta A\|-\|\Delta A\|^2\geq \frac{34}{7}\sigma_j\|\Delta A\|,
\end{align*}
as well as
\begin{align*}
    \sigma_j^2-\sigma_{r+1}^2-|\widetilde\sigma_j^2-\sigma_j^2|&\geq \sigma_j^2-\sigma_{r+1}^2-2\sigma_j\|\Delta A\|-\|\Delta A\|^2\\
    &\geq\sigma_j^2-\sigma_{r+1}^2-\|\Delta A\|(2\sigma_j+\|\Delta A\|)\\
    &\geq \sigma_j^2-\sigma_{r+1}^2-\frac{1}{7}(\sigma_j-\sigma_{r+1})\cdot\frac{15}{7}(\sigma_j+\sigma_{r+1})\\
    &\geq\frac{2}{3}(\sigma_j^2-\sigma_{r+1}^2).
\end{align*}
With probability at least $1-\frac{4}{4^kn^3}$,
\[ \|w\|\leq \frac{3}{2}c_2\sigma\frac{\sqrt{r}\|u_i^T\alpha_{22}\|}{\sigma_r^2-\sigma_{r+1}^2}(\frac{4}{17})^k(\sqrt{\log n}+\sqrt{r}+k).
\]
By Theorem \ref{thm:main2}, we can see that there are $2^{k+1}$ terms in the expansion of $V_2^T\widetilde V_1$ that has order $k$, hence by \eqref{eq:exp_sum}, with probability at least $1-\frac{16}{n^3}$,
\[
\|u_i^T\alpha_{22}V_2^T\widetilde V_1\widetilde\Sigma_1^{-1}\|\leq\sum_{k=0}^\infty 3c_2\sigma\frac{\sqrt{r}\|u_i^T\alpha_{22}\|}{\sigma_r^2-\sigma_{r+1}^2}(\frac{8}{17})^k(\sqrt{\log n}+\sqrt{r}+k)\leq C\sigma\frac{(\sqrt{r\log n}+r)\|u_i^T\alpha_{22}\|}{\sigma_r^2-\sigma_{r+1}^2}.
\]
By the union bound, with probability at least $1-\frac{16}{n^2}$,
\begin{align*}
\|U_2U_2^T\Delta A V_2V_2^T\widetilde V_1\widetilde\Sigma_1^{-1}\|_{2,\infty}&\leq C\sigma\frac{(\sqrt{r\log n}+r)\|\alpha_{22}\|}{\sigma_r^2-\sigma_{r+1}^2}\\&=C\sigma\frac{(\sqrt{r\log n}+r)\|{\Delta A}\|}{(\sigma_r-\sigma_{r+1})(\sigma_r+\sigma_{r+1})}\leq C\sigma\frac{\sqrt{r\log n}+r}{\sigma_r+\sigma_{r+1}}.
\end{align*}
Next, we consider $\|U_2\Sigma_2 V_2^T\widetilde V_1\widetilde\Sigma_1^{-1}\|_{2,\infty}$. Following the same reasoning, we have with probability at least $1-\frac{16}{n^2}$,
\[
\|U_2\Sigma_2 V_2^T\widetilde V_1\widetilde\Sigma_1^{-1}\|_{2,\infty}\leq C\sigma\frac{(\sqrt{r\log n}+r)\|\Sigma_2\|}{\sigma_r^2-\sigma_{r+1}^2}= C\sigma\frac{(\sqrt{r\log n}+r)\sigma_{r+1}}{\sigma_r^2-\sigma_{r+1}^2}.
\]
Combining the above two bounds,
\begin{align*}
\max\{\|U_2\Sigma_2 V_2^T\widetilde V_1\widetilde\Sigma_1^{-1}\|_{2,\infty},\|U_2U_2^T\Delta A V_2V_2^T\widetilde V_1\widetilde\Sigma_1^{-1}\|_{2,\infty}\} & \leq C\sigma(\frac{(\sqrt{r\log n}+r)\sigma_{r+1}}{\sigma_r^2-\sigma_{r+1}^2}+\frac{\sqrt{r\log n}+r}{\sigma_r+\sigma_{r+1}}) \\
& \leq C\sigma\frac{\sqrt{r\log n}+r}{\sigma_r-\sigma_{r+1}}.
\end{align*}
\end{proof}

\end{document}